\documentclass{article}

\usepackage[a4paper, total={7in, 8in}]{geometry}
\usepackage[utf8]{inputenc}   
\usepackage[T1]{fontenc}      
\usepackage{multicol}
\usepackage{amsmath,amsthm} 
\usepackage{amssymb,mathrsfs} 
\usepackage{amssymb}
\usepackage{amsfonts}
\usepackage[normalem]{ulem}
\usepackage{nicefrac}
\usepackage{graphicx}
\usepackage{enumerate}
\usepackage{graphicx} 
\usepackage{lmodern} 
\usepackage{mathtools}
\usepackage{dsfont}
\usepackage{physics}
\usepackage{subcaption}
\usepackage{color}

\newcommand\N{\mathbb{N}}

\newcommand\R{\mathbb{R}}

\newcommand\Sc{\mathcal{S}}

\newcommand\E{\mathbb{E}}

\newcommand\X{\mathcal{D}}
\newcommand\Lc{\mathcal{L}}
\newcommand\dagg{^{\dagger}}

\newcommand\Cinfty{\mathcal{C}^{\infty}}

\newcommand\intX{\int_{\X}}

\newcommand\A{\mathcal{A}}
\newcommand\At{\mathcal{A}^W}
\renewcommand\O{\mathrm{O}}
\renewcommand\dag{^\dagger}

\renewcommand\S{\mathcal{S}}
\newcommand\Sw{\S_W}
\newcommand\Rm{\mathcal{R}}

\newcommand\Qdt{Q_{\Dt}}
\newcommand\Tdt{T_{\Dt}}
\newcommand\Udt{U_{\Dt}}
\newcommand\Qtdt{\widetilde{Q}_{\Dt}}

\newcommand\Id{\mathrm{Id}}
\newcommand\K{\mathcal{K}}
\newcommand\nuw{\nu_W}

\newcommand\hw{h_W}
\newcommand\nuwdt{\nu_{W,\Dt}}
\newcommand\nuwtdt{\tilde{\nu}_{W,\Dt}}
\newcommand\Dt{\Delta t}
\newcommand\e{\mathrm{e}}
\newcommand\PX{\mathcal{P}(\X)}
\newcommand\muinfty{\mu_{\infty}}
\newcommand\lambdadt{\lambda_{\Dt}}
\newcommand\lambdatdt{\tilde{\lambda}_{\Dt}}
\newcommand\C{C} 
\newcommand\Linfty{B^{\infty}(\X)}
\newcommand\hwdt{\hat{h}_{W,\Dt}}

\newcommand\TV{_{\mathrm{TV}}}
\newcommand\MX{\mathcal{M}(\X)}
\newcommand\Mz{\mathcal{M}_0(\X)}
\newcommand\Nit{N_{\mathrm{iter}}}
\newcommand\hh{\hat{h}_W}
\newcommand\nuhw{\hat{\nu}_W}
\newcommand\Shw{\hat{\mathcal{S}}_W}
\newcommand\ic{\mathrm{i}}
\newcommand\chidt{\chi_{\Dt}}

\newcommand{\vertiii}[1]{{\left\vert\kern-0.25ex\left\vert\kern-0.25ex\left\vert #1 
    \right\vert\kern-0.25ex\right\vert\kern-0.25ex\right\vert}}

\DeclarePairedDelimiter\floor{\lfloor}{\rfloor}

\renewcommand{\leq}{\leqslant}

\renewcommand{\geq}{\geqslant}

\newcommand{\dps}{\displaystyle}

\newtheorem{theorem}{Theorem}
\newtheorem{lemma}{Lemma}
\newtheorem{assumption}{Assumption}
\newtheorem{prop}{Proposition}
\newtheorem{corollary}{Corollary}
\newtheorem{remark}{Remark}

\graphicspath{{fig/}}

\begin{document}

\title{Error estimates on ergodic properties of discretized Feynman--Kac semigroups}
\author{Grégoire Ferré and Gabriel Stoltz\\
\small Université Paris-Est, CERMICS (ENPC), Inria, F-77455 Marne-la-Vallée, France \\
}

\date{\today}

\maketitle


\abstract{
We consider the numerical analysis of the time discretization of Feynman--Kac semigroups associated with diffusion processes. These semigroups naturally appear in several fields, such as large deviation theory, Diffusion Monte Carlo or non-linear filtering. We present error estimates à la Talay--Tubaro on their invariant measures when the underlying continuous stochastic differential equation is discretized; as well as on the leading eigenvalue of the generator of the dynamics, which corresponds to the rate of creation of probability. This provides criteria to construct efficient integration schemes of Feynman--Kac dynamics, as well as a mathematical justification of numerical results already observed in the Diffusion Monte Carlo community. Our analysis is illustrated by numerical simulations.
}

\section{Introduction}

The study of Feynman--Kac semigroups for stochastic differential equations (SDEs) has been a topic of growing importance in the past two decades, since these dynamics are related to several theoretical and applied areas of mathematics. They can be seen as standard SDEs whose paths are reweighted according to the exponential of the time integral of some weight function. 

Feynman--Kac semigroups naturally appear in large deviation theory, where they can be used to enhance the likelihood of observing rare fluctuations and henceforth computing cumulant generating functions~\cite{varadhan1984large,den2008large}. They also have important practical applications, such as in the Diffusion Monte Carlo (DMC) method~\cite{foulkes2001quantum}, which is a probabilistic way of estimating the ground state energy of Schr\"odinger operators; or in computational statistics, in particular in (non-linear) filtering~\cite{del2004feynman,douc2007limit}, where relevant trajectories are selected from observations. 

We focus in this paper on the bias arising from the time discretization of the underlying 
continuous stochastic dynamics and of the time integrated weight. Our interest resides in the 
ergodic properties of the discretization, namely the invariant measure as well as the average 
rate of creation of probability. Let us briefly present our setting and results. We study a 
system $q_t\in \X$ evolving in a $d$-dimensional space, assumed to be compact (the extension to 
unbounded spaces poses non-trivial issues, as discussed at various places later on). Typically, 
$\X = \mathbb{T}^d$ (with $\mathbb{T} = \mathbb{R}\backslash\mathbb{Z}$) is a $d$-dimensional
torus.  For convenience, we consider that the evolution is dictated by a stochastic differential
equation with additive noise:
\begin{equation}
\label{eq:overdamped}
d q_t = b(q_t) \, dt+ \sigma \, d B_t,
\end{equation}
where $b:\X\to \R^{d}$ is a $\Cinfty(\X)$ vector field, $\sigma > 0$ and $B_t$ is a standard 
$d$-dimensional Brownian motion. Note that the dynamics~\eqref{eq:overdamped} may be non-reversible. 
Our results can be extended to dynamics with multiplicative noise upon appropriate modifications. 
The infinitesimal generator of the dynamics~\eqref{eq:overdamped}, defined on the 
core $\S = \Cinfty(\X)$, reads
\begin{equation}
\label{eq:generatoroverdamped}
\Lc =  b \cdot \nabla + \frac{\sigma^2}{2} \Delta,
\end{equation}
and we denote by $\Lc\dag$ the adjoint of $\Lc$ on $L^2(\X)$ endowed with the Lebesgue measure. 
Since $\X$ is compact and $b$ is smooth, \eqref{eq:overdamped} admits a unique invariant distribution, denoted by $\nu$, which is solution to the Fokker--Planck equation 
\[
\Lc\dag \nu =0,
\]
see \textit{e.g.}~\cite{bellet2006ergodic,karatzas2012brownian,lelievre2016partial}. Denoting
by $\PX$ the set of probability measures over $\X$, Feynman--Kac type semigroups associated
with a given weight function $W:\X\to\R$ evolve an initial probability measure $\mu\in\PX$ as
follows: for any test function $\varphi\in\S$,
\begin{equation}
\label{eq:feynmankac}
\Phi_t^W(\mu)(\varphi)=\frac{ \E_{\mu}\left[ \varphi(q_t) \, \e^{ \int_0^t W(q_s)\, ds} \right]}
{\dps \E_{\mu}\left[\e^{ \int_0^t W(q_s)\, ds} \right]},
\end{equation}
where the expectations run over initial conditions~$q_0$ distributed according to~$\mu$ and
all realizations of~\eqref{eq:overdamped}.
The family of mappings $\{\Phi_t^W\}_{t \geq 0}$ is a measure-valued non-linear semigroup in the sense that $\Phi_t^W: \PX\to\PX$ depends non-linearly on the initial condition
and, for all $\mu\in\PX$ and $t,s\in\R_+$, 
$\Phi_t^W(\Phi_s^W(\mu))=\Phi_{t+s}^W(\mu)$. Such semigroups have been studied for
a long time in the context of Diffusion Monte Carlo 
(DMC)~\cite{grimm1971monte,anderson1975random,ceperley1980ground,umrigar1993diffusion,foulkes2001quantum} 
in order to estimate the principal eigenvalue of Schrödinger type operators $-\Delta + W$, 
which correspond in our case to $b\equiv 0$. They also appear in the
large deviation community~\cite{giardina2006direct,lecomte2007numerical, tailleur2008simulation,touchette2009large,nemoto2016population,nemoto2017finite} 
where they are related to the principal eigenvalue of $\Lc +W$, which is the dual of the 
rate function -- a result known as the Donsker-Varadhan 
formula~\cite{donsker1975variational,varadhan1984large,den2008large,touchette2009large,dembo2010large}. 
Other fields such as non-linear filtering, Hidden Markov Models~\cite{jasra2015behaviour,douc2007limit,douc2014long} 
and free energy computation~\cite{Jarzynski97PRL,Jarzynski97PRE,rousset2006equilibrium,lelievre2010free}
also motivate the study of such semigroups. 

As discussed in Section~\ref{sec:continuous}, the semigroup~\eqref{eq:feynmankac} converges in general to the average of $\varphi$ with respect to a tilted measure $\nuw$. More precisely, the operator $\Lc\dagg +W$ has a largest eigenvalue $\lambda$ which is isolated from the remainder of the spectrum and non-degenerate, with associated eigenfunction $\nuw$, and 
\[
\Phi_t^W(\mu)(\varphi)\xrightarrow[t\to\infty]{}\, \intX \varphi\, d\nuw
\]
exponentially fast. We address in this work the time discretization of the semigroup~\eqref{eq:feynmankac} using a finite timestep~$\Dt$. The underlying continuous evolution~\eqref{eq:overdamped} is discretized by a Markov chain $(q^n)_{n \in \N}$ and~\eqref{eq:feynmankac} is approximated as (using a simple quadrature rule for the time integral)
\begin{equation}
\label{eq:typicaldiscr}
\Phi_{\Dt,n}^W(\mu)(\varphi)=
\frac{\dps \E_{\mu}\left[ \varphi(q^n) \, \e^{ \Dt \sum_{i=0}^{n-1}W(q^i)} \right] }
{\dps \E_{\mu}\left[\e^{ \Dt \sum_{i=0}^{n-1}W(q^i)} \right] }.
\end{equation}
Under mild assumptions on the discretization scheme (made precise in Section~\ref{sec:discrete}), the discrete semigroup~\eqref{eq:typicaldiscr} converges to an invariant measure $\nuwdt$ in the following sense: for any test function $\varphi \in \S$,
\[
\Phi_{\Dt,n}^W(\mu)(\varphi)\xrightarrow[n \to +\infty]{}\, \intX \varphi \, d\nuwdt.
\]
The core of our work consists in making precise the difference between $\nuw$ and $\nuwdt$. 
We aim in particular at designing numerical schemes leading to the smallest possible biases. 
Although a series of papers study the statistical error of 
estimators such as~\eqref{eq:typicaldiscr} 
(see~\cite{del2001stability,del2003particle,del2004feynman,rousset2006continuous,rousset2006control}), 
there are, to our knowledge, no available estimates on the bias of the limiting measure with 
respect to~$\Dt$. However, in the context of DMC (where we recall $b=0$), it was numerically 
observed that some discretizations provide first, second or fourth order of convergence 
in~$\Dt$ for the largest eigenvalue~$\lambda$ of $\Lc+W$, see for 
example~\cite{anderson1975random,umrigar1993diffusion,mella2000time,sarsa2002quadratic},
and~\cite{el2007diffusion} for the numerical analysis in a simple case. 
The results presented in this paper provide a mathematical justification of such convergences, 
while extending them to the case $b\neq 0$. Let us also mention that Hairer and Weare have 
studied in~\cite{hairer2014improved,hairer2015brownian} the convergence with respect to the timestep of
discretized dynamics similar to the one we consider, over a finite time and for a finite population of replicas. 
They obtain in the limit $\Dt\to 0$ a limiting process, the so-called Brownian fan.

We rely on the techniques developped since the works of Talay and 
Tubaro~\cite{talay1990expansion,talay1990second}, taking advantage of the analytical tools developed 
in a series of papers~\cite{mattingly2010convergence,debussche2012weak,abdulle2012high,abdulle2014high,bou2010long,leimkuhler2016computation,lelievre2016partial}, in order to 
provide a systematic framework to study the bias in the timestep. More precisely, we show 
in Theorem~\ref{theo:numana} that there exist an integer~$p\geq 1$ and a function~$f$ solution 
to a Poisson equation (both depending on the numerical scheme at hand and the quadrature rule 
for the integral), such that, for all~$\varphi \in \S$,
\begin{equation}
\label{eq:typicalresult}
\intX \varphi\, d\nuwdt = \intX \varphi\, d\nuw + \Dt^p\intX \varphi f\, d\nuw +\O(\Dt^{p+1}).
\end{equation}
This result is very similar to those of weak backward error analysis on invariant probability measures of ergodic processes, see for example~\cite{debussche2012weak,bou2010long,leimkuhler2016computation}. Moreover, as the computation of the principal eigenvalue $\lambda$ of the operator $\Lc +W$ is one of the main concerns in Feynman--Kac techniques, we provide in Theorem~\ref{theo:average} the following error estimate:
\begin{equation}
\label{eq:typicalestimate}
\lambdadt:= \frac{1}{\Dt}\log\left[\intX \Qdt^W\mathds{1}\, d\nuwdt \right]
=\lambda + C\Dt^p + \O(\Dt^{p+1}),
\end{equation}
where $\Qdt^W$ is the evolution operator of the discretized dynamics with weight 
function~$W$. This result is interesting since it allows to justify the use of population 
dynamics methods for discretizations of diffusion processes, 
see~\cite{giardina2006direct,tailleur2008simulation,nemoto2016population} for rare events 
simulations, and~\cite{foulkes2001quantum} for DMC. Let us mention that, while the proof 
of~\eqref{eq:typicalresult} relies on previous works concerning error estimates on the invariant 
measure~\cite{talay1990expansion,talay1990second,bou2010long,debussche2012weak,leimkuhler2016computation}, 
the novelty of this work lies in taking into account the non-probability conserving feature of the dynamics. With this point of view and, odd as it may seem, formula~\eqref{eq:typicalresult} appears as a \emph{consequence} of~\eqref{eq:typicalestimate}, and not conversely. An interpretation of this fact is that, in order to prove an error estimate on the invariant probability measure of this non probability-conserving dynamics, we must first show that the discretized process creates or destroys probability at a rate correct up to terms small in $\Dt$.

\medskip  

The paper is organized as follows. Section~\ref{sec:convergence} is devoted to general properties 
of Feynman--Kac semigroups and their discretizations. We then present in Section~\ref{sec:numana}
our main results concerning the numerical analysis of the error on the invariant probability measure, 
depending on the choice of the discretization scheme, before providing numerical 
applications in Section~\ref{sec:application}. Finally, Section~\ref{sec:extensions} proposes 
possible extensions to this work. The proofs of the most technical results are gathered in Section~\ref{sec:proofs}.

\section{Convergence properties of Feynman--Kac semigroups}
\label{sec:convergence}

We present in this section the setting of our study. In particular, we remind convergence 
results and some useful properties of continuous Feynman--Kac semigroups in 
Section~\ref{sec:continuous}, as well as convergence results for their discretizations in 
Section~\ref{sec:discrete}. Although these results are known, we believe that it is useful 
to gather them here to allow for a self-contained presentation of the numerical analysis
framework developped in Section~\ref{sec:numana}. 

\subsection{Continuous dynamics}
\label{sec:continuous}

We denote by $P_t$ the evolution operator associated with the process $(q_t)_{t\geq 0}$
in~\eqref{eq:overdamped}: for all $\mu\in\PX$ and $\varphi\in\S$,
\[
P_t(\mu)(\varphi) = \E_{\mu}\left[ \varphi(q_t) \right].
\]
Its weighted counterpart is
\[
P_t^W(\mu)(\varphi) = \E_{\mu}\left[ \varphi(q_t) \, \e^{\int_0^t W(q_s) \, ds} \right].
\]
The infinitesimal generators of $P_t$ and $P_t^W$ are respectively $\Lc$ and $\Lc +W$, where we
denote with some abuse of notation by~$W$ the multiplication operator by the function~$W$.
Whether a statement corresponds to the function $W$ or the associated
multiplication operator should be clear from the context. We assume in the sequel that
the function $W$ is smooth, so that the associated multiplication operator
stabilizes the core~$\S$.

The existence of a spectral gap for the generator $\Lc +W$ and its adjoint is a key ingredient for our study. Here and in the sequel, and otherwise explicitly mentioned, all operators are considered on the Hilbert space 
\[
L^2(\nu) = \left\{ \varphi \textrm{ measurable} \, \left| \, \intX |\varphi|^2  d\nu < + \infty \right. \right\}.
\]
For a given closed operator~$T$ on~$L^2(\nu)$, we denote by $T^*$ the adjoint of~$T$ 
in~$L^2(\nu)$. In particular,
\[
\forall\, (\varphi,\psi) \in \S^2, \qquad \intX (T\varphi)\psi \, d\nu = \intX \varphi \left(T^*\psi\right) \, d\nu.
\]
In this functional framework, the reversibility of the dynamics is equivalent to the
self-adjointness of~$\Lc$ on $L^2(\nu)$. We however do not assume that this is the case, and
this is why we need to distinguish between eigenelements of $\Lc$ and $\Lc^*$. We can then state
the following.

\begin{prop}
  \label{as:eigen}
  The operator $\Lc + W$, considered on $L^2(\nu)$, has a real isolated principal eigenvalue $\lambda$ with associated eigenfunction $\hh \in \S$ normalized as
  \begin{equation}
    \label{eq:righteigen}
    (\Lc +W) \hh = \lambda \, \hh, \qquad \intX \hh \, d\nu=1.
  \end{equation}
  The operator $\Lc^*+W$ then also admits $\lambda$ as a real isolated principal eigenvalue, with associated eigenfunction $h_W \in \S$ normalized as 
  \begin{equation}
    \label{eq:eigenrelation}
    (\Lc^* +W) h_W = \lambda \, h_W, \qquad \intX h_W \, d\nu = 1. 
  \end{equation}
  Moreover, the functions $\hh$ and $h_W$ are positive.
\end{prop}

The fact that $\hh, h_W \in \S$ is a consequence of elliptic regularity. Let us emphasize that, as a consequence of~\eqref{eq:eigenrelation}, the measure 
\[
\nu_W = h_W \, \nu
\]
is the only invariant probability measure for the evolution encoded by $P_t^{W-\lambda}$. Moreover, when 
the underlying diffusion is reversible, \textit{i.e.} $b = -\nabla V$ and $\nu(dq) = Z^{-1} \, 
\e^{-2 V(q)/\sigma^2} \, dq$, the operator $\Lc$ is self-adjoint ($\Lc^* = \Lc$) so that 
$\hh = h_W$. When $W=0$, it simply holds $h_W = \mathds{1}$ whatever~$b$.

\begin{proof}
It is shown in~\cite{gartner1977large} that the operator $\Lc+W$ has a 
real isolated principal eigenvalue when considered as an operator on $C^0(\X)$, the space of 
continuous functions over~$\X$. This can be proved using the Krein--Rutman
theorem~\cite{du2006order}.  On the other hand, standard results of spectral 
theory of elliptic operators on bounded domains show that $\Lc+W$ on~$L^2(\nu)$ has a discrete 
spectrum, which is bounded above~\cite{reed1978iv}. The first eigenvalue cannot be degenerate 
since the associated eigenvectors are smooth by elliptic regularity and are therefore also 
eigenvectors of $\Lc+W$ 
considered as an operator on $C^0(\X)$. Finally, the positivity of $\hh$ and $h_W$ follows 
from the fact that the evolution semigroup $P_t^W$ and its adjoint are operators with smooth 
and positive transition kernels (since the noise is non-degenerate), together with the 
equalities $P_t^W \hh = \e^{\lambda t} \hh$ and $(P_t^W)^* h_W = \e^{\lambda t} h_W$.
\end{proof}

In what follows, we use the subspaces $L_W^2(\nu)$ and $\Sw$ of functions of average $0$ 
with respect to $\nuw$:
\[
L_W^2(\nu)=\left\{ 
\varphi\in L^2(\nu) \ \middle| \  \intX \varphi\,  d\nuw =0 
\right\},
\qquad
\Sw=\left\{ 
\varphi\in\S \ \middle| \ \intX \varphi\,  d\nuw =0
\right\}.
\]
We also introduce the measure $\nuhw=\hh \,\nu$, the space
\[
\label{eq:spaces2}
\Shw=\left\{ 
\varphi\in\S \ \middle| \ \intX \varphi\,  d\nuhw =0
\right\},
\]
and we denote by
\begin{equation}
  \label{eq:deltaW}
  \delta_W=\inf \Big\{ \lambda-\mathrm{Re}(z),\, z\in\sigma(\Lc + W)\setminus \{\lambda\} \Big\}>0
\end{equation}
the spectral gap of $\Lc +W$ in $L^2(\nu)$.
The fact that the largest eigenvalue $\lambda$ is a priori non-zero corresponds to a possible 
creation ($\lambda>0$) or destruction ($\lambda <0$) of probability induced by the source term 
$W$, which plays the role of an importance sampling function. The 
statement about the spectral gap in Proposition~\ref{as:eigen} implies the convergence of 
the Feynman--Kac semigroup~\eqref{eq:feynmankac}, as stated in the following result. 

\begin{prop}
\label{theo:continuousconvergence}
There exists $C>0$ such that, for all $\mu\in\PX$ and $\varphi\in L^2(\nu)$,
\begin{equation}
\label{eq:feynmankaccv}
\forall \, t \geq 1, 
\qquad 
\left| \Phi_t^{W}(\mu)(\varphi) - \intX\varphi\, d \nuw \right|\leq C \left\|\varphi\right\|_{L^2(\nu)} \e^{-\delta_W t},
\end{equation}
where~$\delta_W$ is defined in~\eqref{eq:deltaW}.
\end{prop}

As made clear in the proof of this result (see Section~\ref{sec:proof_continuous}), it is possible to consider any observable $\varphi \in L^2(\nu)$ even if $\mu$ is singular. This is due to the regularizing properties of the underlying diffusion for positive times, and explains why the convergence result is stated only for times $t \geq 1$. The next proposition will be frequently used in this work. 

\begin{prop}
  \label{prop:intlambda}
  It holds
  \begin{equation}
    \intX W \,d\nuw = \lambda.
  \end{equation}
\end{prop}

\begin{proof}
Integrating both sides of~\eqref{eq:eigenrelation} on~$\X$,
\[
\intX W \, d\nuw = \intX \lambda\hw \, d\nu - \intX \Lc^* h_W \, d\nu
= \lambda \intX h_W \, d\nu - \intX \Lc\mathds{1} \, d\nu_W = \lambda,
\]
since $\Lc\mathds{1}=0$.
\end{proof}

A natural corollary of Propositions~\ref{theo:continuousconvergence} and~\ref{prop:intlambda} is that the largest eigenvalue of $\Lc+W$ can be obtained by a long time average of $W$ using the Feynman--Kac semigroup~\eqref{eq:feynmankac}.

\begin{corollary}
\label{cor:eigencv}
There exists $C>0$ such that, for any initial distribution $\mu \in \PX$,
\[
\left| \Phi_t^W(\mu)(W) -\lambda \right| \leq C \e^{- \delta_W t}.
\]
\end{corollary}

Another important consequence of Proposition~\ref{as:eigen} is the invertibility of the 
generator and its adjoint over suitable functional spaces. 

\begin{prop}
\label{prop:inversibility}
The operator $\Lc + W - \lambda$ is invertible on $\Sw$, in the sense that, for any $g \in \Sw$,
the Poisson equation
\[
(\Lc + W - \lambda)u = g
\]
admits a unique solution $u \in \Sw$, which is denoted by $(\Lc + W - \lambda)^{-1}g$. Similarly, 
$\Lc^* + W - \lambda$ is invertible on~$\Shw$. 
\end{prop}

The proof of this result can be read in Section~\ref{sec:proof_continuous}. Let us emphasize that the smoothness of $W$ is crucial for this proposition to be true. Note also that the stability of the core of the operator $\Lc +W$ would be harder to prove for non-compact state spaces, as this is already a non-trivial statement for the Poisson equation with $W=0$, see~\cite{kopec2014weak,kopec2015weak}.


\subsection{Discretization}
\label{sec:discrete}
We now turn to the discretization of the Feynman--Kac semigroup~\eqref{eq:feynmankac}.
We first define discretization schemes, and show that they are ergodic for some limiting
measure  under mild assumptions. We  also recall the stationarity equation satisfied by this
invariant probability measure, which proves crucial for the numerical analysis developped in 
Section~\ref{sec:numana}.

The properties of discretized Feynman--Kac semigroups are related to the properties of the 
underlying discrete dynamics. The approximation of the continuous 
dynamics~\eqref{eq:overdamped} is given, for a time time $\Dt$, by a Markov chain 
$(q^n)_{n\in\N}$ such that $q^n \simeq q_{n\Dt}$. This Markov chain is characterized by the evolution operator $\Qdt$ defined as
\begin{equation}
\label{eq:Qdt}
(\Qdt\varphi)(q)=\E \left[ \left. \varphi(q^{n+1}) \, \right| \, q^n = q \right].
\end{equation}
A typical example is the Euler Maruyama scheme defined by:
\begin{equation}
  \label{eq:euler}
  q^{n+1} = q^n + b(q^n)\Dt + \sigma \sqrt{\Dt}\, G^n,
\end{equation}
where $(G^n)_{n \geq 0}$ is a familly of independent and identically distributed standard $d$-dimensional Gaussian random variables. In order to perform our analysis in Section~\ref{sec:numana}, it is convenient to rephrase discretizations of~\eqref{eq:feynmankac} such as~\eqref{eq:typicaldiscr} in terms of an evolution operator. For instance, we see that, defining
\begin{equation}
\label{eq:QW1}
(\Qdt^W\varphi)(q)=\e^{\Dt W(q)} (\Qdt\varphi)(q),
\end{equation}
the discretization~\eqref{eq:typicaldiscr} reads, for an initial measure $\mu$ and a
test function $\varphi$,
\begin{equation}
  \label{eq:generaldiscr} 
  \Phi_{\Dt,n}^W(\mu)(\varphi) 
  =\frac{\mu\left( (\Qdt^W)^{n}\varphi \right)}{\mu\left( (\Qdt^W)^{n}\mathds{1} \right)}.
\end{equation}

We use the definition~\eqref{eq:generaldiscr} for more general discretizations of~\eqref{eq:feynmankac} characterized by an evolution operator $\Qdt^W$. Consistency requirements on $\Qdt^W$ are made precise in Assumption~\ref{as:dlsemigroupW} below. This allows us to take into account various integration rules, both for the underlying dynamics and the exponential weights. For instance, the choice
\begin{equation}
\label{eq:QW2}
(\Qdt^W\varphi)(q)=\e^{\frac{\Dt}{2} W(q)} \left[\Qdt\left(\e^{\frac{\Dt}{2} W}\varphi
\right)\right](q),
\end{equation}
well-known in the diffusion Monte Carlo community~\cite{sarsa2002quadratic,mella2000time,makri1989exponential,umrigar1993diffusion}, defines the following semigroup:
\[
\Phi_{\Dt,n}^W(\mu)(\varphi)=
\frac{\dps \E_{\mu}\left[ \varphi(q^n) \, \e^{ \Dt \sum_{i=0}^{n-1}\frac{W(q^i)+W(q^{i+1})}{2}} \right] }
{\dps \E_{\mu}\left[\e^{ \Dt \sum_{i=0}^{n-1}\frac{W(q^i)+W(q^{i+1})}{2}} \right] }.
\]

\begin{remark}
  The weighted evolution on the position $q_t$ can be equivalently formulated as the 
  unweighted evolution for the augmented system~$(q_t,z_t)_{t\geq 0}$, where $z_t\geq 0$ is
  solution to
  \[
  dz_t = z_t W(q_t)\,dt,\quad z_0 =1.
  \]
  However, $z_t$ is unbounded and may diverge to $+\infty$. The augmented
  dynamics~$(q_t,z_t)_{t\geq 0}$ therefore does not have an invariant measure in general,
  which complicates the analysis of the long time limit. Moreover,
  a naive discretization like the Euler-Maruyama scheme applied to~$(q_t,z_t)_{t\geq 0}$ reads
  \[
  \left\{
  \begin{aligned}
    q^{n+1} & =  q^n + b(q^n)\Dt + \sigma \sqrt{\Dt}\, G^n,
    \\ z^{n+1} & =  z^n + z^n W(q^n)\Dt.
    \end{aligned}
  \right.
  \]
  Observe that the positivity of~$z_t$ may not be preserved during the dynamics if~$\Dt$ is
  too large, which is crucial for the numerical scheme to be well-defined. This issue persists
  in general for other schemes.  On the other hand, if $q^n$ is fixed, the process $z_t$  solving 
  \[
  dz_t = z_t W(q^n)\,dt, \quad z_n = z
  \]
  over a time step~$\Dt$ admits the exact solution
  \[
  z^{n+1} = z\, \e^{ W(q^n)\Dt}.
  \]
  Therefore, a first order splitting between $q_t$ and $z_t$ leads to the first order
  integrator~\eqref{eq:QW1}. If we perform a second order
  splitting between $q_t$ and $z_t$, we are back to the second order integration rule
  prescribed by~\eqref{eq:QW2}. As a result, although considering an extended system
  $(q_t,z_t)_{t\geq 0}$ of course makes sense, we see that, in order for the positivity of
  $z_t$ to be unconditionally preserved, we are naturally led to the same schemes as for the
  usual Feynman--Kac dynamics. There is finally a technical restriction with the reformulation
  of the Feynman--Kac dynamics using the augmented process~$(q_t,z_t)_{t\geq 0}$. The
  generator $\Lc_{\mathrm{aug}}$ of $(q_t,z_t)_{t\geq 0}$ is defined, for a test function~$\varphi$,
  through $\Lc_{\mathrm{aug}}\varphi(q,z) = \Lc \varphi(q,z) + zW(q)\partial_z\varphi(q,z)$.
  However, the numerical analysis presented in Section~\ref{sec:numana} uses stability properties of
  the inverse of the generator of the dynamics (see Assumption~\ref{as:stability} below).
  While~$\Lc$ is invertible as an operator acting on functions of~$q$, it is
  much more  difficult to define the inverse of~$\Lc_{\mathrm{aug}}$ in a general way (think of
  the case $W= 0$).

\end{remark}

In what follows, given that the discrete semigroup defines a measure-valued dynamics, we write 
for simplicity $\mu_n=\Phi_{\Dt,n}^W(\mu)$, and we denote by
$\Linfty=\{\varphi\,\mathrm{measurable}\,|\,\sup_{q\in\X}|\varphi(q)|<+\infty\}$ the space of bounded measurable
functions. For a given bounded operator $Q$ on $\Linfty$ and a probability measure
$\mu\in\PX$, we also denote by $\mu Q$ the probability measure defined as
\begin{equation}
  \label{eq:def_mu_Q}
  \forall\, \varphi \in \S, \qquad (\mu Q)(\varphi) = \mu(Q\varphi). 
\end{equation}
We start by recalling a one-step formulation of the non-linear dynamics $(\mu_n)_{n \geq 0}$, as
suggested \textit{e.g.} in~\cite{del2003particle}. This formulation is the basis for a stationarity 
property fundamental in our numerical analysis.

\begin{lemma}
\label{lem:measuredynamics}
The sequence of probability measures $\mu_n=\Phi_{\Dt,n}^W(\mu)$ satisfies the following dynamics:
\[
\mu_{n+1}=\K\mu_n,
\]
where
\begin{equation}
\label{eq:etadynamics}
\forall\, \mu \in\PX,\quad \forall\, \varphi\in\S, 
\qquad 
\K\mu(\varphi) = \frac{\mu\left( \Qdt^W \varphi \right)}{\mu\left(\Qdt^W \mathds{1} \right)}.
\end{equation}
\end{lemma}

\begin{proof}
The proof relies on a simple rewriting: for all $\varphi\in\S$,
\[
\mu_{n+1}(\varphi) 
= \frac{\mu\left( (\Qdt^W)^{n+1}\varphi \right) }{ \mu\left( (\Qdt^W)^{n+1}\mathds{1}\right)} 
= \frac{\mu\left( (\Qdt^W)^{n}(\Qdt^W\varphi) \right) }{\mu\left((\Qdt^W)^{n}\mathds{1}\right)}
\times \frac{\mu\left((\Qdt^W)^{n}\mathds{1}\right)}{\mu\left( (\Qdt^W)^{n}(\Qdt^W\mathds{1})
  \right)} =\frac{\mu_n\left(\Qdt^W\varphi\right)}{\mu_n\left( \Qdt^W\mathds{1} \right)}, 
\]
which gives the result.
\end{proof}

Let us now prove that the measure-valued dynamical process~\eqref{eq:etadynamics} admits a limit measure $\mu_{\infty}$ independent of the initial distribution $\mu$, and that the long time average~\eqref{eq:generaldiscr} converges to the average with respect to this measure. We follow the strategy of Del Moral and collaborators~\cite{del2001stability,del2003particle,del2000branching,del2004feynman}, which relies on the Dobrushin ergodic coefficient of a relevant operator (see Appendix~\ref{sec:dobrushin}). For this, we use the following assumption, which is typically satisfied for discretizations associated with the continuous dynamics~\eqref{eq:feynmankac} on the torus.

\begin{assumption}
  \label{as:minorization}
  The operator $\Qdt^W$ satisfies a minorization and boundedness condition: there exist $\varepsilon\in(0,1)$ and $\eta\in\PX$ such that, for all non-negative bounded measurable function $\varphi$,
  \begin{equation}
    \label{eq:minorization}
    \forall\,q\in\X, \quad
    \varepsilon 
    \eta(\varphi) \leq \left(\Qdt^W \varphi\right)(q) \leq \varepsilon^{-1}\eta(\varphi). 
  \end{equation}
\end{assumption}

The condition~\eqref{eq:minorization} is satisfied for the evolution operator~\eqref{eq:QW1} as 
soon as a condition similar to~\eqref{eq:minorization} is satisfied for the evolution 
operator~$\Qdt$. The latter condition is, in turn, easily seen to be true for the numerical 
scheme~\eqref{eq:euler}, with $\eta(dq) = |\X|^{-1} dq$ the normalized Lebesgue measure 
on~$\X$, see~\cite[Section 3.3.2]{lelievre2016partial}. Similar considerations allow 
to prove that~\eqref{eq:minorization} holds for 
more complicated discretization strategies~\cite{hairer2011yet,lelievre2016partial}.

We can now recall an important result which ensures the existence and uniquess of the 
limiting measure for the discretized Feynman--Kac dynamics. Its proof, taken 
from~\cite{del2001stability}, is recalled in Section~\ref{sec:feynmankacdiscr}. To state the 
result, we introduce the total variation distance between two measures $\mu$, $\nu\in\PX$:
\[
\label{eq:deftotalvariation}
\|\mu - \nu \|_{\mathrm{TV}} = \underset{A\subset \X}{\sup}\ \, | \mu(A) - \nu(A)|, 
\]
where the supremum runs over measurable subsets of~$\X$. Recall that $\PX$ is complete for this distance. 

\begin{theorem}
\label{theo:feynmankacdiscr}
Suppose that Assumption~\ref{as:minorization} holds true. Then the non-linear dynamics~\eqref{eq:etadynamics} admits a unique stationary probability measure $\muinfty$ which is independent of the initial measure and which is a fixed point of $\K$: 
\begin{equation}
\label{eq:fixedpoint}
\muinfty = \K \muinfty.
\end{equation}
Moreover, for any initial distribution $\mu\in\PX$, 
\begin{equation}
\label{eq:TVconvergence}
\| \mu_n - \muinfty \|_{\mathrm{TV}} \leq 2 \left(1 - \varepsilon^2\right)^n.
\end{equation} 

\end{theorem}

\begin{remark}
\label{rem:minorization}
Let us emphasize that the prefactor $\varepsilon$ in~\eqref{eq:minorization} 
typically scales as $\Dt^{-\frac{d}{2}} \exp(-C_L/\Dt)$ for some constant $C_L > 0$.
Indeed, consider for instance the first order discretization~\eqref{eq:euler}.
Its transition kernel between $q$ and $q'$ reads
\[
\Qdt(q,dq')=\frac{1}{(2\pi \sigma^2\Dt)^{\frac{d}{2}}} \mathrm{exp}\left( -\frac{ (q' - q - b(q)\Dt)^2}
{2 \sigma^2 \Dt} \right)dq'.
\]
We then see that $\varepsilon$ scales at dominant order in~$\Dt$ as
$\Dt^{-\frac{d}{2}} \exp(-C_L/\Dt)$ for some constant $C_L > 0$ depending on~$\sigma$
and $\X$, independently on the drift $b$.
Thus, the choice of integrator should not affect significantly the value of~$\varepsilon$.
Note also that, if~$\Qdt^W$ satisfies a uniform version of~\eqref{eq:minorization} with an
additional strong Feller condition, it is possible to derive uniform in~$\Dt$ convergence estimates,
see~\cite[Section~3.3]{ferre2018more}.
\end{remark}

As a consequence of Theorem~\ref{theo:feynmankacdiscr}, if we define a discretization of the 
Feynman--Kac semigroup~\eqref{eq:feynmankac} satisfying Assumption~\ref{as:minorization}, the 
discrete dynamics~\eqref{eq:generaldiscr} admits an invariant probability measure solution to the fixed 
point equation~\eqref{eq:fixedpoint}. We denote by $\nuwdt$ this invariant probability measure to emphasize 
its dependence on both $W$ and the timestep $\Dt$. In view of~\eqref{eq:fixedpoint} and~\eqref{eq:etadynamics}, this measure satisfies the following stationarity equation:
\begin{equation}
\label{eq:stationarity}
\forall\, \varphi \in\S, \quad \intX \Qdt^W\varphi \, d\nuwdt =
\left( \intX \Qdt^W\mathds{1}\, d\nuwdt\right)\left( \intX \varphi\, d\nuwdt\right).
\end{equation}
In particular, if we define the approximate eigenvalue $\lambdadt$ by
\begin{equation}
\label{eq:lambdadt}
\e^{\Dt \lambdadt}= \intX \Qdt^W\mathds{1} \, d\nuwdt,
\end{equation}
then~\eqref{eq:stationarity} can be rewritten as:
\begin{equation}
\label{eq:shortstationarity}
\forall\, \varphi \in\S, \quad \intX\left[ \left(\frac{\Qdt^W -\e^{\Dt \lambdadt}}{\Dt}\right) 
\varphi\right] d\nuwdt = 0.
\end{equation}
This is the stationarity equation of the discretized process upon which the analysis in Section~\ref{sec:numana} is built. Let us emphasize that it involves the approximate eigenvalue $\lambdadt$ accounting for the rate of creation of probability of the discretized process, which differs in general from the largest eigenvalue $\lambda$ of the operator $\Lc+W$ (which accounts for the rate of creation of probability for the continuous process). The numerical analysis of the approximation $\lambdadt$ of $\lambda$ plays an important role in Section~\ref{sec:numana}.

\begin{remark}
In the case $W\equiv 0$, the measure $\nuwdt=\nu_{\Dt}$ is the invariant probability measure of the discretized process without reweighting, and the evolution operator $\Qdt$ conserves probability.
This also implies that $\lambdadt=0$. Therefore~\eqref{eq:shortstationarity} simplifies as
\[
\forall\, \varphi \in\S, 
\qquad 
\intX \left[\left(\frac{\Qdt -1}{\Dt}\right) \varphi \right]d\nu_{\Dt} = 0,
\]
which is the standard stationarity equation of the invariant probability measure for discretizations
of SDEs~\cite{leimkuhler2016computation,lelievre2016partial}. This is because the largest eigenvalue of the discretized evolution operator~$\Qdt$ is~$1$, as for the continuous semigroup~$P_t$. 
\end{remark}

\section{Numerical analysis of the discretization}
\label{sec:numana}

We now turn to the main section of the paper, where we quantify how close $\nuwdt$, the ergodic 
measure for the discrete Feynman--Kac dynamics, is from $\nuw$, the ergodic measure for its 
continuous counterpart. We also make precise the difference at leading order in~$\Dt$. Following
a general strategy to study the error on the invariant probability measure of discretizations
of stochastic processes dating back to~\cite{talay1990expansion} (see 
also~\cite{debussche2012weak,leimkuhler2016computation} as well as the 
review~\cite{lelievre2016partial} for recent accounts), we compare the evolution operator 
$\Qdt^W$ with the Feynman--Kac semigroup $\e^{\Dt(\Lc +W)}$. Although the non 
probability-conserving feature of the dynamics is an additional difficulty, we obtain in 
Section~\ref{sec:mainresults} results similar to those 
of~\cite{talay1990expansion,debussche2012weak,abdulle2014high,lelievre2016partial} concerning 
the error on the invariant probability measure. Moreover, we provide in
Section~\ref{sec:alternative} error bounds for estimators of the eigenvalue~$\lambda$. Finally,
we show how to relate the invariant probability measures of different schemes in
Section~\ref{sec:TU} and discuss in Section~\ref{sec:secondorder} how the Feynman--Kac
discretization essentially inherits the properties of the discretization of the underlying
unweighted dynamics.

\subsection{Error estimates on the invariant probability measure}
\label{sec:mainresults}

\subsubsection{Expansions of the discrete evolution operators $Q_{\Dt}^W$}
\label{sec:expdiscr}

For unweighted dynamics ($W=0$), consistency assumptions on the evolution operator $\Qdt$ characterizing the discretization rely on an expansion of~$\Qdt$ in powers of~$\Dt$ (see the presentation in~\cite{lelievre2016partial}). More precisely, it is assumed that there exist an integer $p\geq 1$ and differential operators $(\A_k)_{k=1,\hdots,p+1}$ such that the evolution operator $\Qdt$ of the discrete dynamics admits the following expansion: for all $\varphi\in\S$,
\begin{equation}
\label{eq:dlsemigroup}
Q_{\Dt}\varphi = \varphi + \Dt \A_1\varphi + \Dt^2\A_2\varphi +\hdots + \Dt^{p+1}\A_{p+1}\varphi +\Dt^{p+2}\Rm_{\Dt}\varphi.
\end{equation}
The differential operators $\A_k$ have finite order and smooth coefficients: for any 
$k \in \{1,\dots,p+1\}$, there exist $m_k \in \mathbb{N}$ and a familly of smooth functions 
$(a_{\alpha})_{|\alpha|\leq m_k}$ (with $\alpha=(\alpha_1,\hdots,\alpha_d)\in\N^d$) such that 
\begin{equation}
  \label{eq:differential_operators}
  \A_k = \sum_{|\alpha|\leq m_k} a_{\alpha} \partial^\alpha,
\end{equation}
where $\partial^\alpha = \partial_{q_1}^{\alpha_1} \dots \partial_{q_d}^{\alpha_d}$. 
Moreover, $\Rm_{\Dt}$ is an operator uniformly bounded in $\Dt$ in the following sense: there exist $\Dt^*>0$, $c>0$ and $m\in\N$ such that
\begin{equation}
\label{eq:uniformbound}
\forall\, \Dt \in (0,\Dt^*], \quad \forall\, \varphi\in\S,
\qquad 
\|\Rm_{\Dt} \varphi \|_{C^0} \leq c\, \|\varphi\|_{\C^m}, 
\end{equation} 
where
\begin{equation}
\label{eq:Cq}
\|\varphi\|_{\C^m} = \underset{|\alpha|\leq m} {\sup}\ \ \underset{q\in\X}{\sup}\ 
\left|\partial^\alpha \varphi(q)\right|.
\end{equation}
The assumptions~\eqref{eq:dlsemigroup} and~\eqref{eq:uniformbound} are standard for the numerical analysis of ergodic measures of SDEs~\cite{talay1990expansion,kopec2014weak,debussche2012weak,abdulle2012high,leimkuhler2016computation,lelievre2016partial}, and are satisfied for a wide range of explicit and implicit schemes defined on compact domains. A scheme is of weak order~$p$ when~\eqref{eq:dlsemigroup} holds with 
\[
\forall\, k \in \{1,\dots,p\}, \qquad \A_k = \frac{\Lc^k}{k!},
\]
see for instance~\cite{MT04}. Typically, $\A_1 = \Lc$ for any reasonable discretization scheme. 

Besides weak and strong errors, 
another notion of consistency is the error arising on the invariant probability measure, in situations 
when the Markov chain associated with~$\Qdt$ admits an invariant probability 
measure~$\nu_{\Dt}$. The error between averages with respect to~$\nu$ and~$\nu_{\Dt}$ are of 
order at least $\Dt^p$ when the scheme is weakly consistent of order~$p$. It can however be of 
higher order~$\Dt^{p'}$ (with $p' \geq p+1$) when  
\begin{equation}
  \label{eq:condition_avg_inv_meas_higher_order}
  \forall\, k \in \{1,\dots,p'\}, \quad \forall\, \varphi \in \S, \qquad \intX \A_k\varphi \, d\nu = 0. 
\end{equation}
This condition is satisfied by operators which are proportional to powers of $\Lc$. See however~\cite{AVZ15,leimkuhler2016computation} for examples of situations where $\A_k$ is not a power of~$\Lc$ but the above condition is met.

In the context of Feynman--Kac averages~\eqref{eq:feynmankac} where we consider approximations $\Qdt^W$ of $\e^{\Dt(\Lc+W)}$, we generalize the conditions~\eqref{eq:dlsemigroup} and~\eqref{eq:uniformbound} as follows. 

\begin{assumption}
\label{as:dlsemigroupW}
There exist an integer $p\geq 1$ and differential operators $(\At_k)_{k=1,\hdots,p+1}$ of the form~\eqref{eq:differential_operators} such that the evolution operator $\Qdt^W$ of the Feynman--Kac dynamics admits the following expansion: for all $\varphi\in\S$,
\begin{equation}
  \label{eq:dlsemigroupgeneral}
  Q_{\Dt}^W\varphi = \varphi + \Dt \At_1\varphi + \Dt^2\At_2\varphi
  +\hdots + \Dt^{p+1}\At_{p+1}\varphi +\Dt^{p+2}\Rm_{W,\Dt}\varphi,
\end{equation}
where $\Rm_{W,\Dt}$ is a uniformly bounded remainder in the sense of~\eqref{eq:uniformbound}. We also assume that $\At_1$ is such that
\begin{equation}
  \label{eq:consistentintegration}
  \At_1 = \A_1 + W, \qquad \A_1 \mathds{1} = 0,
\end{equation}
where $\A_1$ is a differential operator. In particular, $\At_1 \mathds{1} = W$.
\end{assumption}

Let us provide an example of such an expansion when $\Qdt^W$ is defined by~\eqref{eq:QW1}. 

\begin{lemma}
\label{prop:dlfeynmankac}
Assume that~\eqref{eq:dlsemigroup} and~\eqref{eq:uniformbound} hold with $\A_1 \mathds{1} = 0$, and define $\Qdt^W=\e^{\Dt W}\Qdt$. Then Assumption~\ref{as:dlsemigroupW} holds with, for all $k\in\{1,\hdots,p+1\}$,
\begin{equation}
  \label{eq:atilde}
  \At_k\varphi  = \sum_{m=0}^k \frac{W^m}{m!} \A_{k-m} \varphi.
\end{equation}
\end{lemma}

\begin{proof}
The equality follows by expanding the exponential and taking the product with the semigroup 
expansion: there exist $\Dt^*$ and $K>0$ such that
\[
\begin{aligned}
  Q_{\Dt}^W\varphi & = \left( 1 + \Dt W + \frac{\Dt^2}{2} W^2 + \hdots + \frac{\Dt^{p+1}}{(p+1)!}W^{p+1}
+\Dt^{p+2} r_{\Dt,W} \right)
\\ & \qquad \times \left( \varphi + \Dt \A_1\varphi + \Dt^2\A_2\varphi +\hdots + \Dt^{p+1}\A_{p+1}\varphi +\Dt^{p+2}\Rm_{\Dt}\varphi \right),
\end{aligned}
\]
with $\|r_{\Dt,W}\|_{C^0} \leq K$ for $0 < \Dt \leq \Dt^*$. Gathering the terms of order $\Dt^k$ leads to~\eqref{eq:atilde} plus a uniformly bounded remainder, which proves the result. 
\end{proof}

Note that, in~\eqref{eq:atilde}, we obtain~$\At_1=\A_1+W$ where $\A_1$ is defined  
in~\eqref{eq:dlsemigroup}. However, there are other ways to construct Feynman--Kac schemes $\Qdt^W$, 
using for instance a splitting strategy. Let us give an example. Assume for instance that the operator $\Lc$ can be split in two parts: $\Lc=L_1+L_2$. We can then define a splitting scheme as $\Qdt=\e^{\Dt L_2}\e^{\Dt L_1}$, and, by discretizing the time integral of $W$ in three parts (using Simpson's rule) intertwinned with $\e^{\Dt L_2}$ and $\e^{\Dt L_1}$, 
\[
\Qdt^W= \e^{\frac{\Dt}{6} W}\e^{\Dt L_2}\e^{\frac{2\Dt}{3} W}\e^{\Dt L_1}\e^{\frac{\Dt}{6}W}.
\]
In this case, we see that the expansion of $\Qdt^W$ cannot be derived from the one for~$\Qdt$ (by a statement similar to~\eqref{eq:atilde}). The evolution operator $\Qdt^W$ nonetheless satisfies Assumption~\ref{as:dlsemigroupW}. 

\subsubsection{Statement of the main result}

Before stating our main theorem, we need to introduce the following technical assumptions.
\begin{assumption}[Stability]
\label{as:stability}
The operators $\A_1 +W -\lambda$ and $\A_1^* +W -\lambda$ are invertible on $\Sw$ and $\Shw$ respectively (in the sense made precise in Proposition~\ref{prop:inversibility}). 
\end{assumption}

In our setting, a crucial step of the proof consists in building an approximation of the eigenvector $\hh$ to solve an approximate eigenvalue problem for the operator $\Qdt^W$. This is an important difference compared to the case $W\equiv 0$, and requires the following assumption.

\begin{assumption}[Spectral consistency]
\label{as:spectralconsistency}
The operator $\A_1+W$, considered on $L^2(\nu)$, admits $\lambda$ as its largest eigenvalue, 
with associated eigenvector $\hh$: 
\[
(\A_1+W)\hh = \lambda \hh.
\]
\end{assumption}

Note that Assumptions~\ref{as:stability} and~\ref{as:spectralconsistency} are immediately met when the 
schemes are weakly consistent, \textit{i.e.} $\A_1=\Lc$, since 
Assumption~\ref{as:stability} is equivalent to Proposition~\ref{prop:inversibility} while 
Assumption~\ref{as:spectralconsistency} follows from Proposition~\ref{as:eigen}. However, it is possible in principle to construct numerical schemes for which $\A_1\neq \Lc$, in which case Assumptions~\ref{as:stability} and~\ref{as:spectralconsistency} should be checked directly. 

We are now in position to state our main result on the numerical discretization of 
Feynman--Kac semigroups, which makes precise error estimates à la Talay-Tubaro in the ergodic
setting.

\begin{theorem}
\label{theo:numana}
Suppose that Assumptions~\ref{as:dlsemigroupW},~\ref{as:stability} and~\ref{as:spectralconsistency} hold. Assume also that the operators $\At_k$ in~\eqref{eq:dlsemigroupgeneral} are such that, for all $k\in\{1,\hdots,p\}$, there exists $a_k\in\R$ with
\begin{equation}
\label{eq:invariance}
\forall\, \varphi \in \S, 
\qquad 
\int_{\X} \left(\At_k \varphi\right)\, d\nuw = a_k \intX \varphi\, d \nuw.
\end{equation} 
Define also $f\in \Sw$ as
\begin{equation}
\label{eq:leading}
f=f_0 - \intX f_0 \, d\nuw, \qquad
\left\{
\renewcommand{\arraystretch}{2.5} 
\begin{array}{l}
\dps (\A_1^*+W - \lambda)(\hw f_0) = \tilde{g}, \\ 
\dps \tilde{g} = -(\At_{p+1})^* \hw +  \hw \frac{\dps \intX \At_{p+1} \hh \, d\nuw}
{\dps \intX\hh d\nuw}\in\Shw.
\end{array}
\right.
\end{equation}
Then, there exist a timestep $\Dt^*>0$ and an operator $R_{W,\Dt}$ (uniformly bounded in $\Dt$ in the sense of~\eqref{eq:uniformbound}) such that, for any $0<\Dt\leq \Dt^*$ and any $\varphi\in\S$,
\begin{equation}
\label{eq:errordt}
\int_{\X} \varphi\, d\nuwdt = \int_{\X} \varphi\, d\nu_W + \Dt^p \int_{\X} \varphi f\, d\nuw
+\Dt^{p+1}R_{W,\Dt}\varphi.
\end{equation}
\end{theorem}

Note that the denominator in the second term on the right-hand side of the definition of~$\tilde{g}$ is positive thanks to Proposition~\ref{as:eigen}. In general, in~\eqref{eq:dlsemigroupgeneral}, we expect $\At_k$ to be $(\Lc+W)^k/k!$ (which corresponds to a scheme of weak order $k$), in which case~\eqref{eq:invariance} holds for $a_k=\lambda^k/k!\neq 0$ (see~\eqref{eq:a_k_lambda^k} below for a proof of the latter equality). This factor comes from the fact that $\Qdt^W$ does not
conserve probability. Indeed, for the evolution
operator $\Qdt$ of a Markovian dynamics, one always has
\[
\intX \Qdt\mathds{1} \, d\nuw = 1.
\]
On the other hand, considering~\eqref{eq:dlsemigroupgeneral} and applying~\eqref{eq:invariance} 
to $\varphi=\mathds{1}$ leads to
\[
\intX \Qdt^W \mathds{1}\, d\nuw = 1 + \Dt\, a_1 + \hdots + \Dt^p a_p +\Dt^{p+1}r_{W,\Dt},
\]
where $r_{W,\Dt}$ is a remainder term which is uniformly bounded for $\Dt$ sufficiently small. This is the manifestation at the discrete level of the fact that, over a timestep $\Dt$, the dynamics increases or decreases approximately the probability mass by a factor $\e^{\Dt \lambda}$. The relation~\eqref{eq:invariance} should be compared to the invariance relation~\eqref{eq:condition_avg_inv_meas_higher_order} for $W=0$.

\subsubsection{Proof of Theorem~\ref{theo:numana}}

The proof of Theorem~\ref{theo:numana} relies on four lemmas which allow to easily conclude the proof. We follow the same strategy as for the error analysis of the invariant probability measure proposed in~\cite{leimkuhler2016computation,lelievre2016partial} but additionnal technical difficulties arise due to the non-linearity of the stationarity equation~\eqref{eq:stationarity}. The first step (Lemma~\ref{lem:approxstationarity}) is to construct the leading correction  term $f$. We next use a projector in Lemma~\ref{lem:step2} to relate the exact stationary measure $\nuw$ and its approximation $\nuwdt$. An a priori estimate on the approximate eigenvalue defined in~\eqref{eq:lambdadt} is then provided in Lemma~\ref{lem:step3}. Finally, an approximate inverse operator is constructed in Lemma~\ref{lem:approxoperator}. In the proofs and also in the statements below, the remainders may change from line to line in the calculation, but we do not change the notation for convenience. There are two types of remainders: terms of the form $R_{W,\Dt}\varphi$ where $R_{W,\Dt}$ is a differential operator satisfying~\eqref{eq:uniformbound}, and functions $r_{W,\Dt}$ such that, for any $k \geq 1$, there is $K>0$ and $\Dt^*$ for which $\|r_{W,\Dt}\|_{C^k} \leq K$ when $0 < \Dt \leq \Dt^*$.

To begin with, we give the expression of the leading correction term $f$. It relies on
an approximate reformulation of~\eqref{eq:stationarity} which leads to an expression similar
to~\eqref{eq:shortstationarity} up to a remainder of order~$\Dt^{p+1}$.

\begin{lemma}
\label{lem:approxstationarity}
Under the assumptions of Theorem~\ref{theo:numana}, for any $\varphi\in\S$,
\begin{equation}
\label{eq:approxstationary}
\intX \left(\Qdt^W\varphi\right)(1+\Dt^p f)\, d\nuw 
=\left( \intX \Qdt^W\mathds{1}(1+\Dt^p f)\,d\nuw\right) \left( \intX 
\varphi\,(1+\Dt^p f)\,d\nuw \right) + \Dt^{p+2} 
R_{W,\Dt}\varphi,
\end{equation}
where $R_{W,\Dt}$ is a uniformly bounded remainder (in the sense of~\eqref{eq:uniformbound}) 
and $f$ is defined in~\eqref{eq:leading}.
\end{lemma}

The proof of this lemma is presented in Section~\ref{sec:approxstationarity}.
Defining the approximate eigenvalue $\lambdatdt$ by
\begin{equation}
\label{eq:tildelambdadt}
\e^{\Dt \lambdatdt}=\intX \Qdt^W\mathds{1}(1+\Dt^p f)\, d\nuw,
\end{equation}
\eqref{eq:approxstationary} can be rewritten as
\begin{equation}
\label{eq:shortapproxstationary}
\intX \left[ \left( \frac{\Qdt^W - \e^{\Dt \lambdatdt}}{\Dt}\right) \varphi \right]
(1 +\Dt^p f)\, d\nuw = \Dt^{p+1}R_{W,\Dt}\varphi.
\end{equation}
This expression allows to identify the leading order correction term $\Dt^p f$ in $\nu_{W,\Dt} - \nu_W$ and can be thought of as the  approximate counterpart of~\eqref{eq:shortstationarity}. The second step is to use a projector that on the one hand stabilizes in $\Sw$ the operator appearing in~\eqref{eq:shortapproxstationary}, and on the other hand relates the exact stationary measure $\nuw$ and its approximation~$\nuwdt$. For this we introduce the following projectors: for all $\phi \in \S$,
\begin{equation}
\label{eq:proj}
\Pi\phi = \phi - \intX \phi\, d\nu, 
\qquad 
\Pi_W\phi = \phi - \intX\phi\, d\nu_W.
\end{equation}
The operator $\Pi$ is the $L^2(\nu)$ orthogonal projector on $L_0^2(\nu)$, while $\Pi_W$ is a projector on $L_W^2(\nu)$ which is not orthogonal for the canonical scalar product on $L^2(\nu)$. However, it is orthogonal on $L^2(\nuw)$, so that, for all $\psi$, $\phi\in\S$,
\begin{equation}
\label{eq;sym}
\intX (\Pi_W \psi) \phi\, d\nuw = \intX \psi (\Pi_W \phi)\, d\nuw.
\end{equation}
We can then show the following result, whose proof can be found in Section~\ref{sec:step2}.

\begin{lemma}
\label{lem:step2}
Under the assumptions of Theorem~\ref{theo:numana}, it holds, for any $\phi\in\S$,
\begin{equation}
\label{eq:step2}
\begin{aligned}
 \intX\left[\Pi_W \left(\frac{\Qdt^W -\e^{\Dt \lambdadt}}{\Dt}\right) 
\Pi_W\phi\right] d\nuwdt  =& \intX\left[\Pi_W \left(\frac{\Qdt^W 
-\e^{\Dt \lambdatdt}}{\Dt}\right) 
\Pi_W\phi\right](1+\Dt^p f)\, d\nuw \\
&   +\Dt^{p+1}R_{W,\Dt}\phi,  
\end{aligned}
\end{equation}
where $R_{W,\Dt}$ is a uniformly bounded remainder in the sense of~\eqref{eq:uniformbound}.
\end{lemma}

Here, we see that two different operators appear inside the integrals because the factors $\e^{\Dt \lambdadt}$ and $\e^{\Dt \lambdatdt}$ are different. The next lemma shows that these quantities are the same up to terms of order $\Dt^{p+2}$. As mentioned earlier, this is an important difference with the analysis in the case $W = 0$. Some a priori estimate on the eigenvalue is required to conclude the proof, whereas, for the unweighted case, the largest eigenvalue of the evolution operator is $1$ with eigenvector $\mathds{1}$ both for the continuous process and its discretization. The proof, provided in Section~\ref{sec:step3}, relies on building an approximate eigenfunction for the operator $\Qdt^W$. Similar estimates were obtained in the Diffusion Monte Carlo context in analytically solvable cases in~\cite{mella2000time}. 

\begin{lemma}
\label{lem:step3}
Under the assumptions of Theorem~\ref{theo:numana}, there exist $\Dt^*>0$, $c>0$ and 
functions $u_1,\hdots,u_p\in\Sw$ such that the function $\hwdt = \hh + \Dt\, u_1+\hdots 
+ \Dt^p\, u_p$ satisfies
\begin{equation}
\label{eq:approxeigenQw}
\Qdt^W \hwdt = \e^{\Dt \lambdatdt}\hwdt + \Dt^{p+2} r_{W,\Dt}, 
\qquad 
\intX \hwdt \, d\nu = 1,
\end{equation}
where $\|r_{W,\Dt}\|_{C^0} \leq c$ for all $0<\Dt \leq \Dt^*$. As a consequence, there exist $\Dt'$ and $C>0$ such that
\begin{equation}
\label{eq:approxeigen}
\e^{\Dt \lambdadt} = \e^{\Dt \lambdatdt} + \Dt^{p+2}\tilde{r}_{W,\Dt},
\end{equation}
with $| \tilde{r}_{W,\Dt} | \leq C$ for all $0<\Dt\leq \Dt'$.
\end{lemma}

Once we have reached this point, it is possible to replace the eigenvalue 
$\e^{\Dt\lambdadt}$ by $\e^{\Dt\lambdatdt}$ in Lemma~\ref{lem:step2}.
The last step is to build an approximate inverse of the operator
\[
\Pi_W \left( \frac{\Qdt^W - \e^{\Dt \lambdatdt}}{\Dt} \right) \Pi_W,
\]
as provided in the next lemma (see Section~\ref{sec:approxoperator} for the proof).

\begin{lemma}
\label{lem:approxoperator}
Under the assumptions of Theorem~\ref{theo:numana}, for any $0<\Dt \leq \Dt^*$, there is an 
operator $S_{\Dt}^W:\Sc \to \Sc$ for which 
\begin{equation}
\label{eq:approxoperator}
\forall\,\varphi\in\Sw, 
\qquad
\Pi_W \left( \frac{\Qdt^W - \e^{\Dt \lambdatdt}}{\Dt} \right) \Pi_W S_{\Dt}^W\varphi = \varphi + \Dt^{p+1}R_{W,\Dt}\varphi,
\end{equation}
where $R_{W,\Dt}:\Sc \to \Sc$ is a uniformly bounded remainder in the sense
of~\eqref{eq:uniformbound}, and $S_{\Dt}^W$ admits the following uniform bounds: for
any $k \geq 0$, there exist $K > 0$ and $m \in \mathbb{N}$ (depending on~$k$) such that 
\[
\forall\, \Dt \in (0,\Dt^*], \qquad \left\|S_{\Dt}^W\varphi\right\|_{\C^k} \leq K \left\|\varphi\right\|_{\C^m}.
\]
\end{lemma}

We now have all the tools to prove Theorem~\ref{theo:numana}. First, plugging the 
estimate~\eqref{eq:approxeigen} obtained in Lemma~\ref{lem:step3} in the error 
expansion~\eqref{eq:step2} obtained in Lemma~\ref{lem:step2} leads to, for any $\phi\in\S$,
\begin{equation}
  \label{eq:laststep}
  \begin{aligned}
    \intX\left[\Pi_W \left(\frac{\Qdt^W -\e^{\Dt \lambdatdt}}{\Dt}\right) 
      \Pi_W\phi\right] d\nuwdt & = \dps \intX\left[\Pi_W \left(\frac{\Qdt^W 
        -\e^{\Dt \lambdatdt}}{\Dt}\right) 
      \Pi_W\phi\right](1+\Dt^p f)\, d\nuw \dps \\
    & \quad +\Dt^{p+1} R_{W,\Dt}\phi,
  \end{aligned}
\end{equation}
where $R_{W,\Dt}$ satisfies~\eqref{eq:uniformbound}. We next consider the approximate inverse 
operator $S_{\Dt}^W$ built in Lemma~\ref{lem:approxoperator}, and set 
$\phi=S_{\Dt}^W\Pi_W \varphi$ in~\eqref{eq:laststep}. Therefore, for any $\varphi \in \S$,
\[
\intX (\Pi_W\varphi)\, d\nuwdt = \intX (\Pi_W \varphi) (1+\Dt^p f)\, d\nuw + \Dt^{p+1} 
\widetilde{R}_{W,\Dt}\varphi=  \Dt^p \intX (\Pi_W \varphi)   f \, d\nuw + \Dt^{p+1} 
\widetilde{R}_{W,\Dt}\varphi,
\]
where $\widetilde{R}_{W,\Dt}$ satisfies~\eqref{eq:uniformbound}. Since $f$ has average $0$ with respect to $\nuw$, this gives
\[
\intX \varphi\, d\nuwdt = \intX \varphi\, d\nuw + \Dt^p \intX \varphi f \, d\nuw +\Dt^{p+1} \widetilde{R}_{W,\Dt}\varphi,
\]
which concludes the proof of Theorem~\ref{theo:numana}.

\subsection{Alternative error estimate for the principal eigenvalue}
\label{sec:alternative}

We present in this section a useful application of Theorem~\ref{theo:numana}, which provides an 
error estimate for the approximation of the principal eigenvalue~$\lambda$ of the 
operator $\Lc + W$. The choice $\varphi \equiv W$ allows to compute this eigenvalue by 
ergodic averages, as shown in Proposition~\ref{prop:intlambda} and Corollary~\ref{cor:eigencv}. As a result, this eigenvalue can be approximated using Theorem~\ref{theo:numana}, whose application to $\varphi\equiv W$ gives
\[
\intX W\, d\nuwdt = \lambda + \Dt^p \intX Wf\, d\nuw + \Dt^{p+1} r_{W,\Dt},
\]
where $r_{W,\Dt}$ is uniformly bounded for $\Dt$ small enough. Although this formula can be 
used in simulations to estimate $\lambda$, we present an error estimate for an alternative 
approximation more commonly used in practice. We will also see in Section~\ref{sec:secondorder} 
that this alternative 
formula can be more accurate than the estimate based on averaging~$W$.

\begin{theorem}
\label{theo:average}
Suppose that Assumption~\ref{as:dlsemigroupW} holds, with a numerical scheme consistent at order $p$ (that is, $\At_k=(\Lc+W)^k/k!$ for $1\leq k \leq p$). Then there exist $\Dt^*>0$ and $C>0$ such that
\begin{equation}
\label{eq:loglambda}
\lambdadt=\frac{1}{\Dt}\log \left[\intX \Qdt^W\mathds{1}\, d\nuwdt \right]
= \lambda +\Dt^{p}\left(\lambda_{p+1}-\frac{\lambda^{p+1}}{(p+1)!}\right) 
+ \Dt^{p+1} r_{\Dt,W},
\end{equation}
with $|r_{\Dt,W}|\leq C$ for any $0< \Dt \leq \Dt^*$, and
\begin{equation}
\label{eq:lambdap1}
\lambda_{p+1}= \intX \At_{p+1}\mathds{1}\,d\nuw +  \intX Wf\, d\nuw.
\end{equation}
\end{theorem}

This result is important since it implies that we can approximate the eigenvalue~$\lambda$ by 
computing $\lambdadt$, which is proportional to the logarithm of the average creation of 
probability over a timestep $\Dt$ (given by $\Qdt^W \mathds{1}$) at stationarity. This is the 
reason why we need the coefficients $a_k$ to be correct up to order $p$ (\textit{i.e.} 
$a_k=\lambda^k/k!$) since they represent the creation of probability of the discretized process. 
The estimate~\eqref{eq:loglambda} justifies the use of population based 
dynamics~\cite{giardina2006direct,tailleur2008simulation,nemoto2016population} when the 
underlying continuous diffusions are discretized in time. We illustrate the error 
estimate~\eqref{eq:loglambda} in the numerical simulations reported in 
Section~\ref{sec:application}. 

\begin{proof}
We use Lemma~\ref{lem:step3} to prove the theorem, which highlights the importance of this result in our context. In all this proof, $r_{W,\Dt}$ denotes a smooth function which may change from line to line, but whose $C^0$ norm is always uniformly bounded for sufficiently small timesteps~$\Dt$. From the definition~\eqref{eq:lambdadt} and the estimate~\eqref{eq:approxeigen}, 
\[ 
\lambdadt = \frac{1}{\Dt}\log\left( \e^{\Dt\lambdadt} \right)
 = \frac{1}{\Dt} \log\left( \e^{\Dt\lambdatdt}+ \Dt^{p+2}r_{W,\Dt} \right).
\]
Expanding $\e^{\Dt\lambdatdt}$ defined in~\eqref{eq:tildelambdadt} in powers of~$\Dt$ and recalling that $\At_1\mathds{1}=W$, 
\[
\begin{aligned}
\lambdadt &  = \frac{1}{\Dt} \log \left[ \intX \left( 1 + \Dt W + \Dt^2 \At_2\mathds{1}
+\hdots + \Dt^{p+1} \At_{p+1}\mathds{1} \right) ( 1 + \Dt^p f)\, d\nuw 
+ \Dt^{p+2}r_{W,\Dt} \right]
\\& = \frac{1}{\Dt} \log \left[ 1 +\Dt \lambda + \hdots + \Dt^p \frac{\lambda^p}{p!} + \Dt^{p+1}\intX\left( \At_{p+1}\mathds{1} + Wf\right) d\nuw + \Dt^{p+2}r_{W,\Dt}\right],
\end{aligned}
\]
where we used that $\intX f\, d\nuw =0$ and, in view of~\eqref{eq:eigenrelation},
\begin{equation}
\label{eq:a_k_lambda^k}
\forall\, k\in\{1,\hdots,p\}, 
\qquad 
\intX \At_k \mathds{1} \, d\nuw = \intX \left[\frac{(\Lc + W)^k}{k!} \mathds{1}\right] \hw \, d\nu = \intX \left[\frac{(\Lc^* + W)^k}{k!} \hw\right] d\nu = \frac{\lambda^k}{k!}.
\end{equation}
Therefore, 
\[
\begin{aligned}
\lambdadt &  = \frac{1}{\Dt} \log \left[ \e^{\Dt \lambda} - \sum_{k=p+2}^{+\infty} \Dt^k\frac{\lambda^k}{k!} + \Dt^{p+1}\intX\left( \At_{p+1}\mathds{1}-\frac{\lambda^{p+1}}{(p+1)!} + Wf\right)d\nuw + \Dt^{p+2}r_{W,\Dt}
\right]
\\&  = \frac{1}{\Dt} \log \left[ \e^{\Dt \lambda}+ \Dt^{p+1}\intX\left( \At_{p+1}\mathds{1}-\frac{\lambda^{p+1}}{(p+1)!} 
+ Wf\right)d\nuw + \Dt^{p+2} r_{W,\Dt} \right].
\end{aligned}
\]
Given that $\e^{\Dt\lambda}$ is uniformly bounded for $0<\Dt\leq \Dt^*$ and equal to $1$ at leading order in~$\Dt$, we obtain, by expanding the logarithm,
\[
\lambdadt = \lambda + \Dt^{p}\,\e^{-\Dt \lambda} \left[\intX\left( \At_{p+1}\mathds{1} + Wf\right)d\nuw -\frac{\lambda^{p+1}}{(p+1)!}\right]+ \Dt^{p+1}r_{W,\Dt}.
\]
The result then follows from $\e^{-\Dt \lambda}= 1 + \Dt \, r_{\lambda,\Dt}$ and the definition~\eqref{eq:lambdap1} of~$\lambda_{p+1}$. 
\end{proof}

\subsection{TU Lemma}
\label{sec:TU}

In the context of splitting schemes, it may be useful to relate the invariant probability measures of two 
numerical schemes differing by the ordering of the applied operators. This is the purpose 
of a result called ``TU lemma'' in~\cite{leimkuhler2016computation}, which we adapt to our 
context in Lemma~\ref{lem:TUW}. We then state a similar version of this lemma for the
eigenvalues of two such schemes in Proposition~\ref{prop:directTU}.
We will see in Section~\ref{sec:secondorder} that this last result can be combined with 
Theorem~\ref{theo:average} to show that the schemes~\eqref{eq:QW1} and~\eqref{eq:QW2} both 
provide second order estimates of the principal eigenvalue $\lambda$ 
using~\eqref{eq:loglambda}, when the discretization of the process $\Qdt$ is weakly consistent 
of order~2. 

\begin{lemma}
  \label{lem:TUW}
  Consider two numerical schemes for the Feynman--Kac dynamics with associated evolution 
operators $\Qdt^W$ and $\Qtdt^W$ satisfying Assumption~\ref{as:minorization}, and denote 
by $\nuwdt$ and $\nuwtdt$ respectively the associated ergodic measures in the sense of 
Theorem~\ref{theo:feynmankacdiscr}. Assume that the evolution operators are related by two 
operators $\Tdt^W$ and $\Udt^W$, bounded on $\Linfty$, as:
  \begin{equation}
    \label{eq:TUrelationW}
    \forall\,n\geq 1, \qquad \left(\Qtdt^W\right)^{n}= \Tdt^W \left(\Qdt^W\right)^{n} \Udt^W.
  \end{equation}
  Then, for any $\varphi\in\S$,
  \begin{equation}
    \label{eq:TUcvW}
    \intX \varphi\, d\nuwtdt = \frac{\dps\intX \Udt^W\varphi\, d\nuwdt}
          {\dps \intX \Udt^W\mathds{1}\,d\nuwdt}.
  \end{equation}
\end{lemma}

For the TU lemma stated in~\cite{leimkuhler2016computation}, the typical case of application 
corresponds to $\Qdt= \Udt \Tdt$ and $\Qtdt =  \Tdt \Udt$, with two Markov operators $\Tdt$ 
and $\Udt$. In this case, the relation~\eqref{eq:TUrelationW} holds with a power $n-1$ on 
the right-hand side. For Feynman--Kac semigroups, $\Tdt^W$ and $\Udt^W$ are a priori such 
that $\Tdt^W\mathds{1}\neq \mathds{1}$ and $\Udt^W\mathds{1}\neq \mathds{1}$. A typical case 
of interest is $\Qdt^W=\e^{\Dt W}\Qdt$ and $\Qtdt^W=\Qdt\left( \e^{\Dt W} \cdot \right)$, in 
which case~\eqref{eq:TUrelationW} is satisfied with $\Tdt^W = \e^{-\Dt W}$ and 
$\Udt^W = \e^{\Dt W}$. 

\begin{proof}
For any $\mu\in\PX$ and any $\varphi\in\S$,
\[
\begin{aligned}
\widetilde{\Phi}_{\Dt,n}(\mu)(\varphi)& = \frac{ \mu \left( \left(\Qtdt^W\right)^{n}\varphi\right)}{ \mu \left( \left(\Qtdt^W\right)^{n}\mathds{1}\right)} = \frac{ \mu \left(\Tdt^W \left(\Qdt^W\right)^{n}\Udt^W\varphi\right)}{ \mu \left( \Tdt^W \left(\Qdt^W\right)^{n}\Udt^W\mathds{1} \right)}
\\ & =  \frac{ \left(\mu\Tdt^W\right)  \left( \left(\Qdt^W\right)^{n}\mathds{1}\right)}{\left(\mu\Tdt^W\right) \left(\left(\Qdt^W\right)^{n}\Udt^W\mathds{1} \right)}
\times 
\frac{\left(\mu\Tdt^W\right) \left(\left(\Qdt^W\right)^{n}\Udt^W\varphi\right)} {\left(\mu\Tdt^W\right) \left(\left(\Qdt^W\right)^{n}\mathds{1} \right)}
= \frac{ \Phi_{\Dt,n}(\mu_1)(\Udt^W\varphi)}{\Phi_{\Dt,n}(\mu_1)(\Udt^W\mathds{1})},
\end{aligned}
\]
where $\mu_1 \in \PX$ is defined by 
\[
\forall \phi \in \S, \qquad \mu_1\left(\phi\right) = \frac{\dps \mu\left(\Tdt^W\phi\right)}{\dps \mu\left(\Tdt^W\mathds{1}\right)}.
\]
The result then follows from the ergodic limits
\[
\underset{n\to+\infty}{\lim} \Phi_{\Dt,n}(\mu_1)(\varphi) = \intX \varphi\, d\nuwdt,
\qquad
\underset{n\to+\infty}{\lim} \widetilde{\Phi}_{\Dt,n}(\mu)(\varphi) = \intX \varphi\, d\nuwtdt,
\]
as provided by Theorem~\ref{theo:feynmankacdiscr}. 
\end{proof}

In our framework, the approximate
principal eigenvalue $\lambdadt$ is another important feature of a discretization scheme. In fact, 
under an additional assumption on the operators $\Tdt^W$ and $\Udt^W$, schemes related 
by~\eqref{eq:TUrelationW} share the same approximate eigenvalues in the sense 
of~\eqref{eq:lambdadt}. This is made precise in the following proposition (see 
Section~\ref{sec:proof_directTU} for the proof).

\begin{prop}
\label{prop:directTU}
Fix a timestep $\Delta t > 0$ and 
consider a numerical scheme for the Feynman--Kac dynamics corresponding to an evolution 
operator $\Qdt^W$ satisfying Assumption~\ref{as:minorization}, with associated invariant 
measure~$\nuwdt$ given by Theorem~\ref{theo:feynmankacdiscr}, and eigenvalue $\lambdadt$
defined by~\eqref{eq:lambdadt}.
Consider next a second scheme  corresponding to an operator $\Qtdt^W$ related to~$\Qdt^W$ 
by~\eqref{eq:TUrelationW},
with operators $\Udt^W$ and $\Tdt^W$ bounded on $\Linfty$ and for which there 
exists $\alpha >0$ such that, for any $\varphi \in\S$ with $\varphi \geq 0$, 
\begin{equation}
\label{eq:minoTUW}
\alpha \varphi \leq \Udt^W \varphi \leq \alpha^{-1}\varphi, \quad 
\alpha \varphi \leq \Tdt^W \varphi \leq \alpha^{-1}\varphi.
\end{equation}
Then, $\Qtdt^W$ satisfies Assumption~\ref{as:minorization}, and its invariant probability measure is denoted
by $\nuwtdt$. Moreover, its associated eigenvalue 
$\lambdatdt$ defined by
\begin{equation}
\label{eq:lambdatdt}
\lambdatdt = \frac{1}{\Dt}\log \left[ \intX \Qtdt^W \mathds{1}\, d\nuwtdt \right],
\end{equation}
is such that
\[
\lambdatdt=\lambdadt.
\]
\end{prop}

The eigenvalue $\lambdatdt$ should not be mistaken in this context with the 
definition~\eqref{eq:tildelambdadt}, which serves as an intermediate in the proof of
Theorem~\ref{theo:numana}. A careful inspection of the proof shows that it would be possible
to consider a slightly different assumption~\eqref{eq:minoTUW}. 

\begin{remark}
Although Proposition~\ref{prop:directTU} may look odd at first sight, it has a natural
interpretation in terms of matrices. Indeed, if $A\in\R^{n\times n}$ and
$B \in \R^{n\times n}$ are two square matrices with nonnegative entries, the products
$AB$ and $BA$ share the same real principal eigenvalue.
One can show this by the following argument. For any matrix $M\in\R^{n\times n}$ with nonnegative
entries, the spectral radius  
\[
\rho(M)=\underset{n\to +\infty}{\lim} \| M^n \|^{\frac{1}{n}}
\]
is an eigenvalue of $M$ (see~\cite{schaefer1974banach}). 
This eigenvalue is the equivalent of the principal eigenvalue for the operator~$\Lc+W$ since
it is the eigenvalue of the matrix~$M$ with the largest real part.   
It is easy to see that $\rho(AB)=\rho(BA)$ by noting that
\[
\rho(AB) =\underset{n\to +\infty}{\lim} \| (AB)^n \|^{\frac{1}{n}}
=\underset{n\to +\infty}{\lim} \| A (BA)^{n-1} B \|^{\frac{1}{n}}
\leq \underset{n\to +\infty}{\lim}\| A\|^{\frac{1}{n}} 
\| (BA)^{n-1} \|^{\frac{1}{n}} \| B\|^{\frac{1}{n}} = \rho(BA). 
\]
This leads to $\rho(AB)\leq \rho(BA)$, and, by symmetry, $\rho(BA)\leq \rho(AB)$;
hence $\rho(AB)= \rho(BA)$. In the same way, evolution operators related 
by~\eqref{eq:TUrelationW} share the same principal eigenvalue even though, a priori, they do 
not admit the same invariant probability measures. The proof of Proposition~\ref{prop:directTU}, 
presented in Section~\ref{sec:proof_directTU}, follows a path similar to the one used here for matrices.
\end{remark}

\subsection{Second order schemes}
\label{sec:secondorder}

We now turn to second order schemes for Feynman--Kac dynamics. They are the most interesting ones in practice, since they can provide an important improvement in the accuracy for a relatively cheap computational overhead. Moreover, in our case, they can be  straightforwardly built from second order schemes for the dynamics~\eqref{eq:overdamped}, as a consequence of Theorem~\ref{theo:numana}.

\begin{lemma}
\label{prop:secondorder}
Suppose that~\eqref{eq:dlsemigroup} and~\eqref{eq:uniformbound} hold with the following expansion for $\Qdt$:
\begin{equation}
\label{eq:secondorder}
\forall\,\varphi\in\S, 
\qquad 
\Qdt\varphi = \varphi + \Dt \Lc\varphi + \Dt^2 \frac{\Lc^2\varphi}{2} + \Dt^3\A_3\varphi + \Dt^4 \Rm_{\Dt}\varphi, 
\end{equation}
where $\A_3$ is a differential operator with smooth coefficients and $\Rm_{\Dt}$ satisfies~\eqref{eq:uniformbound}. Then the operator $\Qdt^W$ defined by
\[
\label{eq:midpoint}
\forall \,\varphi\in\S, \qquad \Qdt^W\varphi= \e^{\frac{\Dt}{2}W} \Qdt\left(\e^{\frac{\Dt}{2}W}\varphi\right), 
\]
satisfies Assumption~\ref{as:dlsemigroupW} with $p=2$:
\begin{equation}
\label{eq:secondorderdl}
\forall \,\varphi\in\S, \qquad \Qdt^W\varphi = \varphi + \Dt (\Lc + W) \varphi +\Dt^2\frac{(\Lc +W)^2\varphi}{2} + \Dt^3\At_3\varphi + \Dt^4 \Rm_{W,\Dt}\varphi,
\end{equation}
where
\[
\At_3\varphi = \A_3\varphi + \frac{W^3\varphi}{6} +\frac{\Lc(W^2\varphi)}{8} + \frac{\Lc^2(W\varphi)}{4} + \frac{W\Lc^2\varphi}{4} + \frac{W\Lc(W\varphi)}{4} + \frac{W^2 \Lc\varphi}{8},
\]
and $\Rm_{W,\Dt}$ satisfies~\eqref{eq:uniformbound}. 
\end{lemma}

The interpretation of this result is the following: when we have a scheme consistent at order 
2 for the dynamics with $W=0$, we immediately obtain a second order scheme for the Feynman--Kac 
dynamics by using the corresponding Markov chain and a trapezoidal rule for the time integral 
in the exponential. Thanks to the consistency at order one ($\At_1=\Lc+W$) and 
Propositions~\ref{as:eigen} and~\ref{prop:inversibility}, the assumptions of 
Theorems~\ref{theo:numana} and~\ref{theo:average}  are immediately satisfied with $p=2$.

\begin{proof}
The expression of~$\At_3$ can be obtained by a direct computation or with the Baker-Campbell-Hausdorff formula (see~\cite{hairer2006geometric}), which is a convenient way to perform the algebra allowing to make precise the various terms in expansions in powers of~$\Dt$. Let us sketch how this is done, and refer to~\cite{leimkuhler2016computation} for strategies of proof in order to make the expansions below rigorous. First, 
\[
\Qdt = \e^{\Dt \Lc} + \Dt^3\left(\A_3-\frac{\Lc^3}{6}\right) + ...
\]
and, by the Baker-Campbell-Hausdorff formula,
\[
\e^{\Dt W/2} \e^{\Dt \Lc} \e^{\Dt W/2} = \e^{S_{\Dt}}, 
\qquad 
S_{\Dt} = \Dt (\Lc+W) + \frac{\Dt^3}{12} \left( -\frac12 \big[W,[W,\Lc]\big] + \big[\Lc,[\Lc,W]\big]\right),
\]
where $[A,B] = AB-BA$ denotes the commutator of two operators $A$ and $B$. Therefore, 
\[
\e^{\Dt W/2} \e^{\Dt \Lc} \e^{\Dt W/2} = \Id + \Dt (\Lc+W) + \frac{\Dt^2}{2} (\Lc+W)^2 + \frac{\Dt^3}{6} (\Lc+W)^3 + \frac{\Dt^3}{12} \left( -\frac12 \big[W,[W,\Lc]\big] + \big[\Lc,[\Lc,W]\big]\right) + ...
\]
The conclusion then follows from
\[
\e^{\Dt W/2} \Qdt \e^{\Dt W/2} = \e^{\Dt W/2} \e^{\Dt \Lc} \e^{\Dt W/2} + \Dt^3\left(\A_3-\frac{\Lc^3}{6}\right) + ...
\]
upon developping the commutators. 
\end{proof}

When we are interested in the computation of the principal eigenvalue with 
Theorem~\ref{theo:average}, we can in fact show that the left-point 
integration~\eqref{eq:QW1} is sufficient for $\lambdadt$ to be correct at order 2 if $\Qdt$ 
is consistent at order 2 (\textit{i.e.}~\eqref{eq:secondorder} holds). In particular, the 
discretization scheme for the Feynman--Kac dynamics need not be consistent at order~2 for 
the eigenvalue to be correct at order~2 (in the same way that the invariant probability measure for 
discretizations of ergodic SDEs can be correct at order~2 even if the discretization itself 
is only weakly consistent at order~1, 
see~\cite{AVZ15,leimkuhler2016computation,lelievre2016partial}). This consequence of 
Proposition~\ref{prop:directTU} is made precise in the following proposition.

\begin{prop}
\label{prop:appliTU}
Consider an evolution operator $\Qdt$ with the following familly of discretizations for 
the Feynman--Kac dynamics:
\[
\Qdt^{W,\delta} = \e^{(1-\delta)W\Dt}\Qdt \e^{\delta W \Dt},\quad \delta \in [0,1].
\]
Suppose that Assumption~\ref{as:minorization} holds for at least one of these schemes, and 
denote by $\lambdadt^{\delta}$ their associated eigenvalues as in~\eqref{eq:lambdadt}. 
Then, $\lambdadt^{\delta}$ is independent of $\delta$. Moreover, when $\Qdt$ 
satisfies~\eqref{eq:secondorder}, the eigenvalue $\lambdadt^{\delta}$ satisfies~\eqref{eq:loglambda}
with $p=2$ for any $\delta\in[0,1]$.
\end{prop}

\begin{proof}
The proof is a simple application of Proposition~\ref{prop:directTU}. Consider the scheme
$\Qdt^{W,\delta}$ for a fixed
$\delta\in[0,1]$ and the scheme $\Qtdt^W=\e^{\Dt\frac{W}{2}}\Qdt\e^{\Dt\frac{W}{2}}$, which
corresponds to a trapezoidal approximation of the integral. We can assume without loss of 
generality that $\Qtdt^W$ satisfies Assumption~\ref{as:minorization}.
Then, $\Qdt^{W,\delta}$ is related to $\Qtdt^W$ through~\eqref{eq:TUrelationW} for the 
corresponding operators: 
\[
\Udt^W=\e^{\left(\delta-\frac{1}{2}\right)\Dt W}, \quad \Tdt^W=\e^{\left(\frac{1}{2}
-\delta\right)\Dt W}.
\] 
The operators $\Udt^W$ and $\Tdt^W$ are bounded on $\Linfty$ and 
satisfy~\eqref{eq:minoTUW} with $\alpha=\e^{-\Dt\|W\|_{\Linfty}/2}$. Therefore, by 
Proposition~\ref{prop:directTU}, the eigenvalue $\lambdadt^{\delta}$ associated to $\Qdt^{W,\delta}$
is equal to $\lambdatdt$, the eigenvalue associated to $\Qtdt^W$, and thus does not
depend on $\delta$. Moreover, by Lemma~\ref{prop:secondorder}, if $\Qdt$
satisfies~\eqref{eq:secondorder}, $\Qtdt^W$ satisfies the assumptions of
Theorem~\ref{theo:average} with $p=2$. This shows that the eigenvalue $\lambdadt^{\delta}$
satisfies~\eqref{eq:loglambda} with $p=2$ whatever the integration rule (\textit{i.e.} for
any $\delta\in[0,1]$).
\end{proof}

\begin{remark}
  Proposition~\ref{prop:appliTU} shows that the eigenvalue~$\lambdadt$ can be correct at
  order two even though the scheme only has weak order one. One may wonder
  whether it is also possible to have second order convergence on the invariant measure
  when $\Qdt$ corresponds to a scheme of weak order one. As mentioned in
  Section~\ref{sec:expdiscr} this is the case when $W=0$, see the examples
  in~\cite{leimkuhler2016computation}. Perturbative arguments for
  small $W$ however show that this extra cancellation on the invariant measure cannot happen for
  a non-constant $W$, see~\cite{ferre2019PhD}.
\end{remark}

\section{Numerical application}
\label{sec:application}

The goal of this section is to illustrate the error estimates presented in 
Section~\ref{sec:numana} on a toy example. For this, we 
consider~\eqref{eq:overdamped} over the one dimensional torus $\X=\mathbb{T}$ with possibly 
non-gradient drifts:
\begin{equation}
\label{eq:oneDoverdamped}
d q_t = (-  V'(q_t) + \gamma) \, dt + \sigma\, dB_t,
\end{equation}
where $V$ is a smooth potential and $\gamma\in\R$. Let us emphasize that a constant force
is not the gradient of a smooth periodic function. We first make precise in 
Section~\ref{sec:algo} the Monte Carlo algorithm used to compute the Feynman--Kac 
averages. We next describe in Section~\ref{sec:Galerkin} a Galerkin method to compute reference values for the
properties of interest. Note that such a discretization method can be used only for low-dimensional systems;
but, when it can be used, it typically provides more accurate results than stochastic methods.
Finally, we present our numerical results in Section~\ref{sec:results}.

\subsection{Monte Carlo discretization}
\label{sec:algo}

\paragraph{Discretization of the underlying SDE.}
The Euler-Maruyama discretization of the dynamics~\eqref{eq:oneDoverdamped} is given by: 
\begin{equation}
\label{eq:EM}
q^{n+1}=q^n + \big(- V'(q^n)+\gamma \big)\Dt + \sigma\sqrt{\Dt}\, G^n,
\end{equation} 
where $G^n$ are independent and identically distributed one-dimensional standard Gaussian
variables. It is well known that this scheme is weakly consistent of order
one (see for instance~\cite{MT04,debussche2012weak}).
In order to test our results on a second order scheme, we use a discretization
proposed \textit{e.g.}
in~\cite{abdulle2012high,zygalakis2011existence,FS17,trstanova2016mathematical}:
\begin{equation}
\label{eq:modified}
q^{n+1}= q^n - V' \left( q^n +\left(- V'(q^n)+\gamma\right) \frac{\Dt}{2}
+\frac{1}{2}\sigma\sqrt{\Dt}\, G^n 
\right) \Dt + \gamma \Dt 
- \frac{\sigma^2}{8} V'''(q^n) \Dt^2 + \sigma \sqrt{\Dt}\, G^n.
\end{equation}
It can be proved that this scheme is of weak order~$2$. 

\paragraph{Weighted dynamics.}
Once the underlying SDE has been discretized, a Monte Carlo
scheme for approximating the associated Feynman--Kac semigroup~\eqref{eq:generaldiscr} has to be devised.
Several methods have been succefully applied in order to compute Feynman--Kac averages, 
generally referred to as \emph{Sequential Monte Carlo} or \emph{Population Monte Carlo} 
methods~\cite{dou2001seq,del2004feynman,lelievre2010free}.
For simplicity and numerical efficiency, we present here a population method with multinomial 
resampling. More precisions on this familly of algorithms are available 
in~\cite{dou2001seq}, see also~\cite[Chapter~6]{lelievre2010free} in the context of
free energy computation and~\cite{hairer2014improved} in the context of Diffusion Monte Carlo.

The algorithm relies on a dynamics run over a set of replicas of the system. 
At each step, the replicas are updated according to the dynamics prescribed by the evolution 
operator $\Qdt$, and are assigned an importance weight depending on the choice of 
discretization rule for the integral. The replicas are then resampled following a multinomial 
distribution with their respective weights, before computing the desired averages.
This technique prevents the variance of the estimator to increase exponentially in time, 
a common problem when computing directly quantities such as~\eqref{eq:generaldiscr}. We now make 
precise the algorithm.

Consider a population of $M$ replicas $(q_m)_{m=1,\hdots,M}$ initially distributed according
to some probability measure $\mu^{\otimes_M}$ over $\X^M$ and evolving through a Markov
kernel $\Qdt$ with timestep $\Dt>0$. We denote by $\chidt:\X\times\X\to\R$ a weight function to be
chosen later on. The algorithm consists in repeating for each time $0\leq n < \Nit$
the following steps:
\begin{enumerate}[(1)]
\item For $m\in\{1,\hdots,M\}$, evolve the $m^{\mathrm{th}}$ replica as
$\tilde{q}_m^{n+1}\sim \Qdt(q_m^n,\cdot)$;
\item Compute the weight of each replica as $w_m^n=\e^{\chidt(q_m^n,\tilde{q}_m^{n+1})}$;
\item Compute the total creation of mass as
\[
P^n = \sum_{m=1}^M w_m^n,
\]
and the normalized probability vector $p^n\in\R^m$ with components $p_m^n= w_m^n /P^n$, for $m\in\{1,\hdots,M\}$;
\item Resample the replicas $(\tilde{q}_m^{n+1})_{m=1}^M$ according to the multinomial 
distribution associated with $p^n$, which defines a new set of replicas
$(q_m^{n+1})_{m=1}^M$;
\item Compute the estimator
\[
\widehat{\varphi}_n= \frac{1}{M} \sum_{m=1}^M \varphi(q_m^{n+1}).
\]
\end{enumerate}

Until now, we did not specify the choice of function $\chidt$, which depends on the
discretization rule for the integral in~\eqref{eq:feynmankac}. In practice, given a discretization of the SDE
characterized by an operator $\Qdt$, we use the 
schemes defined by the left point integration $\e^{\Dt W}\Qdt$, and by the trapezoidal 
integration $\e^{\Dt \frac{W}{2}}\Qdt\e^{\Dt \frac{W}{2}}$. 
They correspond respectively to the choices:
\begin{equation}
\label{eq:chichoice}
\chidt(q,q')=\Dt\, W(q) \quad \mbox{and}\quad  \chidt(q,q')=\Dt \left(\frac{ W(q) + W(q')}{2}
\right).
\end{equation}
The principal eigenvalue of the operator $\Lc + W$ is then estimated 
with~\eqref{eq:loglambda} through
\begin{equation}
\label{eq:lambdaMC}
\lambdadt=\frac{1}{\Dt}\log\left[\intX \Qdt^W\mathds{1}\ d\nuwdt \right]
\approx \frac{1}{\Dt} \log\left[ \frac{1}{\Nit}\sum_{n=0}^{\Nit-1}P^n \right],
\end{equation}
while the average of $\varphi$ is estimated by
\begin{equation}
\label{eq:phiMC}
\intX \varphi \, d\nuwdt \approx \frac{1}{\Nit} \sum_{n=0}^{\Nit-1}\widehat{\varphi}_n,
\end{equation}
where the $\approx$ sign indicates the approximation arising from the finiteness of the
number $M$ of replicas and of the number $\Nit$ of steps.
We do not take these errors into account  and ensure numerically that they are sufficiently
small in our simulations to observe the bias due to the timestep (this bias being quite
small in practice, this also motivates to study a one-dimensional model, see the numerical
results below). The reader interested
in the convergence rates of this type of algorithm when $M\to +\infty$ and $\Nit\to +\infty$ is
refered \textit{e.g.} to~\cite{del2003particle,dou2001seq,rousset2006control,el2007diffusion}.

\subsection{Galerkin discretization}
\label{sec:Galerkin}

We now make precise the Galerkin method that can be used to estimate $\lambdadt$ and 
$\intX \varphi\, d\nuwdt$. This discretization provides reference values for the Monte Carlo
method described in Section~\ref{sec:algo}. In particular, when $V=0$ and 
$\gamma=0$, the two methods should give the same result since the Euler scheme~\eqref{eq:EM} is exact in law in this 
specific case.

\paragraph{Choice of the Galerkin basis.}
Since we work with periodic functions, we consider the Galerkin subspace $\mathrm{Span}\{e_{-N},\hdots,e_N\}$ with
\[
e_n(q)=\e^{2 \ic\pi n q}.
\]
The generator of the SDE~\eqref{eq:oneDoverdamped} reads
\[
\Lc=(-V'+\gamma)\partial_q + \frac{\sigma^2}{2}\partial_q^2.
\] 
The operators $\Lc^\dagger$ and~$W$ are represented in this Galerkin subspace
by the matrices $L^N$, $B^N\in\mathbb{C}^{(2N+1)\times (2N+1)}$ defined as
\[
\forall\, m,n\in\{-N,\hdots,0,\hdots,N\}, \qquad L_{n,m}^N= \intX e_n (\Lc^\dagger e_m), \quad B_{n,m}^N = \intX W e_n e_m .
\]
The value of $N$ is chosen sufficiently large for all results to be converged with respect to 
this parameter. The only source of error in the quantities we compute then arises from the finiteness of the timestep 
$\Dt>0$, and possibly numerical quadratures to evaluate certain integrals. 
Our experience shows that $N=30$ is already sufficient for the applications described in Section~\ref{sec:results}.

\paragraph{References quantities for $\Dt = 0$.}
The invariant probability measure $\nu_W$ satisfies the eigenvalue problem $(\Lc^\dagger + W)\nu_W = \lambda \nu_W$.  
We compute a reference approximation $\lambda^N_0$ to~$\lambda$ by computing the eigenvalue of $L^N+B^N$ with the largest real part:
\[
(L^N + B^N)\mathcal{V}_{W,0}^N = \lambda_0^N \mathcal{V}_{W,0}^N.
\]
The associated eigenvector allows to construct the following approximation of~$\nu_W$:
\[
\nu_{W,0}^{N} = \sum_{k=-N}^N [\mathcal{V}_{W,0}^N]_k e_k.
\]
The normalization condition $[\mathcal{V}_{W,0}^N]_0 = 1$ ensures that $\nu_{W,0}^{N}$ has a total mass~1. Averages of observables~$\varphi$ are then estimated by computing the following integral
\[
\intX \varphi(q) \nu_{W,0}^{N}(q) \, dq
\]
using a one-dimensional quadrature rule. 

\paragraph{Reference quantities for $\Dt > 0$.}
We next approximate the evolution operators of the first order scheme
$\e^{\Dt W} \Qdt$ and of the second order one $\e^{\Dt \frac{W}{2}}\Qdt \e^{\Dt \frac{W}{2}}$, 
respectively as
\begin{equation}
\label{eq:matrices}
Q_{\Dt,1}^{W,N}=\e^{\Dt B^N}\e^{\Dt L^N}, \qquad Q_{\Dt,2}^{W,N}=\e^{\Dt \frac{B^N}{2}}
\e^{\Dt L^N} \e^{\Dt \frac{B^N}{2}}.
\end{equation}
For each value of~$\Dt$, we construct the above matrices, and compute their 
respective principal eigenvalues $\Lambda_{\Dt,1}^N$, $\Lambda_{\Dt,2}^N$ and eigenvectors
$\mathcal{V}_{W,\Dt}^{N,1}$, $\mathcal{V}_{W,\Dt}^{N,2}\in\mathbb{C}^{2N+1}$ by diagonalization
(still with the normalization condition $[\mathcal{V}_{W,\Dt}^{N,j}]_0 = 1$ for $j=1,2$).
We then consider the following approximations of the principal eigenvalue $\lambda$ of the Feynman--Kac operator $\Lc +W$, based on~\eqref{eq:loglambda}:
\begin{equation}
\label{eq:vpappli}
\lambda_{\Dt,1}^N=\frac{1}{\Dt}\log \Lambda_{\Dt,1}^N,
\qquad \lambda_{\Dt,2}^N=\frac{1}{\Dt}\log \Lambda_{\Dt,2}^N.
\end{equation}
Averages of $\varphi$ with respect to the invariant probability measure are approximated by the 
following quantity, using the eigenvectors $\mathcal{V}_{W,\Dt}^{N,1}$ and
$\mathcal{V}_{W,\Dt}^{N,2}$: for $j=1,2$,
\begin{equation}
\label{eq:phiappli}
\intX \varphi(q) \nu_{W,\Dt}^{N,j}(q) \, dq, \qquad  \nu_{W,\Dt}^{N,j} = 
\sum_{k=-N}^N 
[\mathcal{V}_{W,\Dt}^{N,j}]_k e_k.
\end{equation}
In view of Theorem~\ref{theo:numana}, we expect the
average of $\varphi$ to converge linearly in $\Dt$ for the first order scheme when $\Dt\to 0$, 
and quadratically for the second order scheme. We also use the TU-lemma to show 
that, by appropriately correcting the first order scheme, we recover the same results as for the second order scheme.
More precisely, we apply~\eqref{eq:TUcvW} with $\Udt^W=\e^{\Dt \frac{W}{2}}$, which leads to the following approximation of the average (estimated in practice using a numerical quadrature):
\begin{equation}
\label{eq:TUappli}
\frac{\dps \intX \e^{\Dt \frac{W(q)}{2}}\varphi(q)\nu_{W,\Dt}^{N,1}(q) \, dq}
{\dps \intX \e^{\Dt \frac{W(q)}{2}}\nu_{W,\Dt}^{N,1}(q) \, dq}.
\end{equation}
On the other hand, from Proposition~\ref{prop:appliTU}, the eigenvalues $\lambda_{\Dt,1}^N$
and $\lambda_{\Dt,2}^N$ should be equal, and therefore $\lambda_{\Dt,1}^N$ need not be corrected. 

\subsection{Numerical results}
\label{sec:results}

\paragraph{Zero-potential case.}
We first choose $V= 0$, $\sigma=\sqrt{2}$, $W(q) = (\cos(2\pi q))^2$ and 
$\varphi(q)=\exp(\cos(2\pi q))$. As mentioned earlier, in this case, the Euler 
scheme~\eqref{eq:EM} is exact in law, so that the only source of error arises from the
integration of the exponential weight. We consider the dynamics represented
by the operator $\e^{\Dt W}\Qdt$ and $\e^{\Dt \frac{W}{2}}\Qdt\e^{\Dt \frac{W}{2}}$
with $\Qdt=\e^{\Dt \Lc}$, and first compare the results of the Galerkin discretization
discussed in Section~\ref{sec:Galerkin}.  The results reported in Figure~\ref{fig:DMC} 
confirm our predictions: the averages of $\varphi$ converge at first and second order for
the first order and second order Galerkin schemes respectively; while the eigenvalues are the same,
as expected from Proposition~\ref{prop:appliTU}, and so both converge at second order. 
In this case, the numerical method based on~\eqref{eq:loglambda} is therefore more accurate
than the one based on averaging $W$  with~\eqref{eq:errordt} to compute the principal
eigenvalue $\lambda$, which would lead to errors of order~1 in the timestep (numerical results
not shown here).

\begin{figure}[h]
  \begin{subfigure}[h]{0.4\textwidth}  
    \includegraphics[width=\textwidth,angle=270]{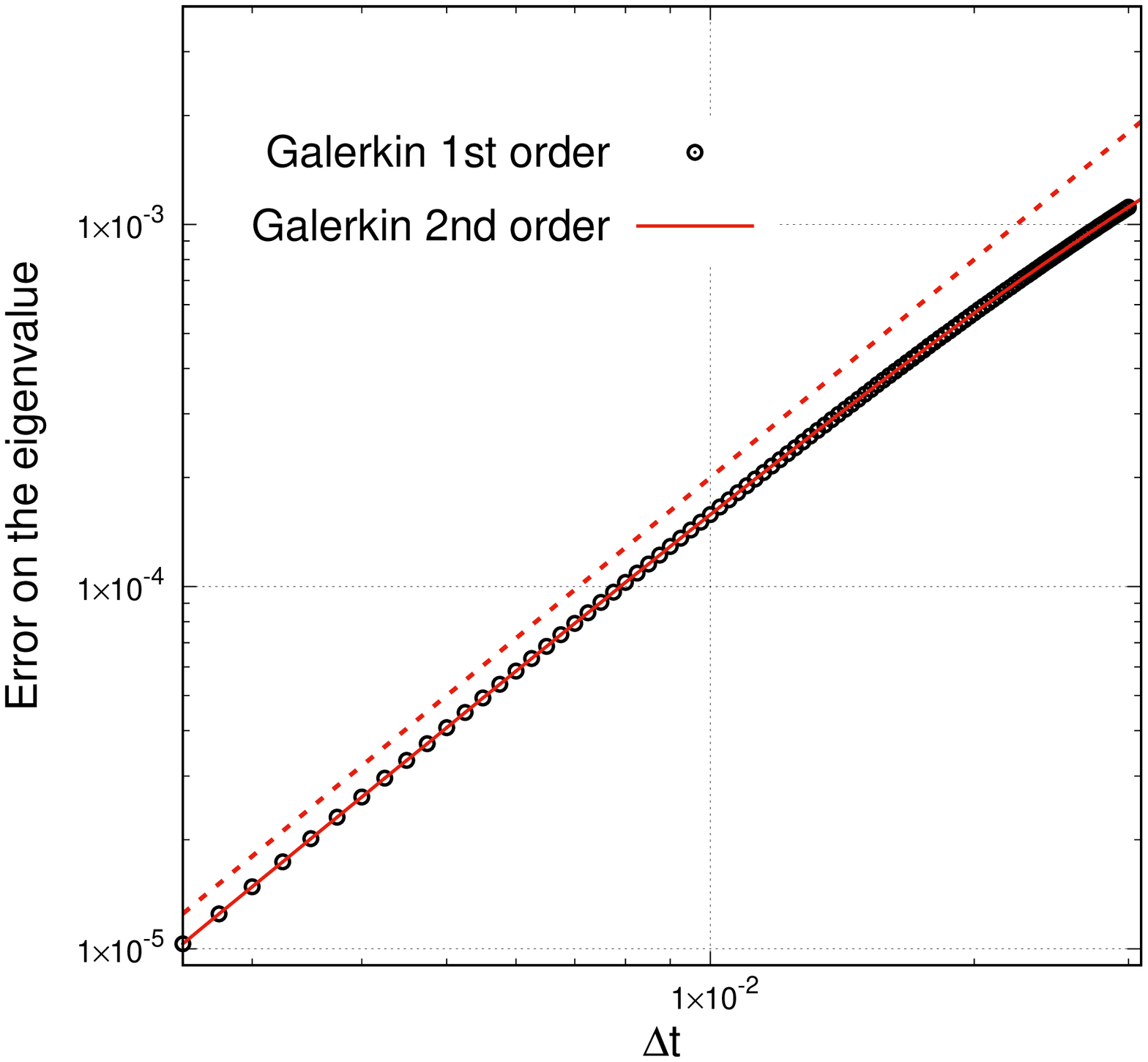}
    \caption{Convergence of the eigenvalue.}
  \end{subfigure}
  \qquad  \hspace{0.5cm}
  \begin{subfigure}[h]{0.4\textwidth}
    \includegraphics[width=\textwidth,angle=270]{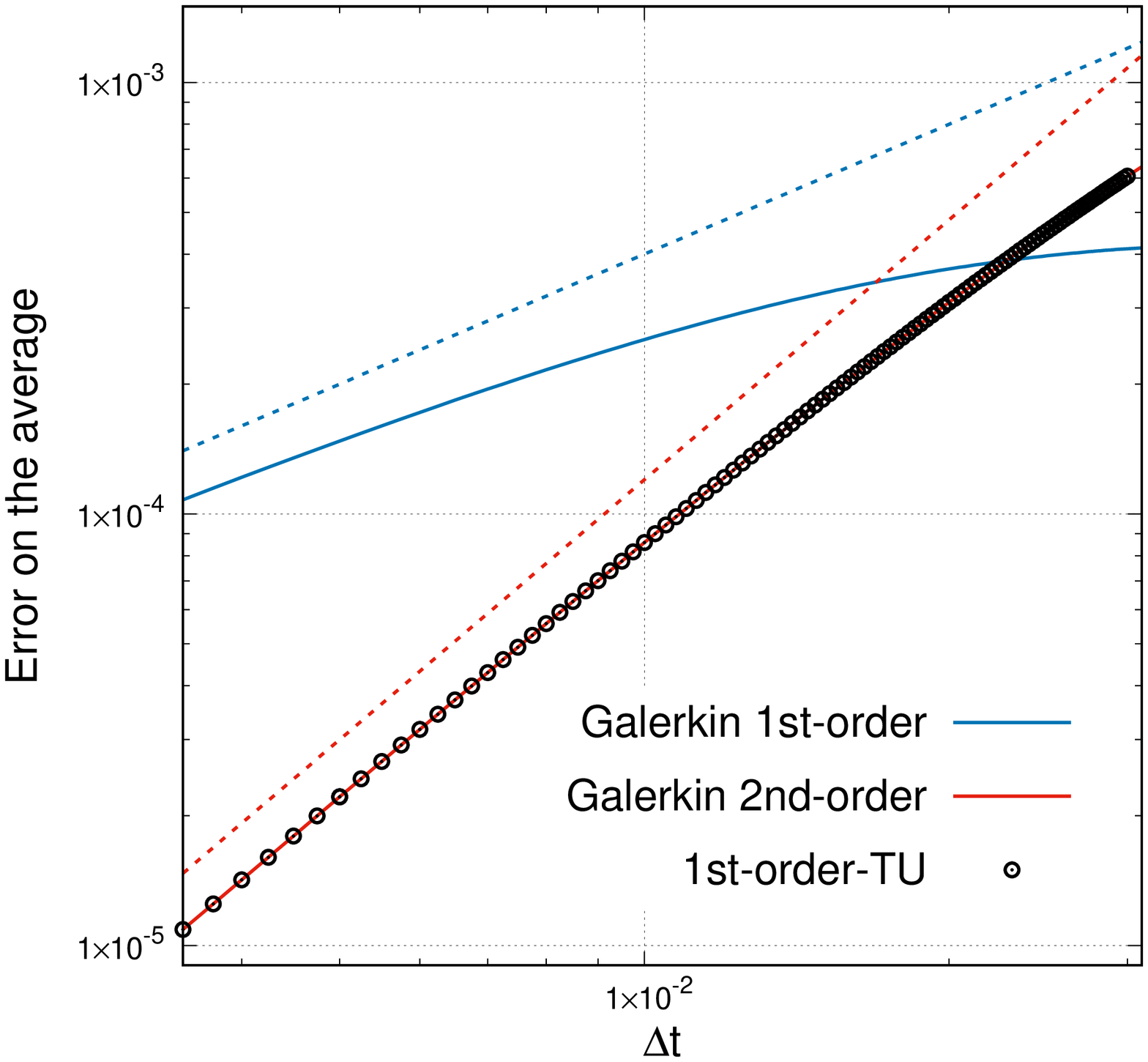}
    \caption{Convergence of the average of $\varphi$. }
  \end{subfigure}
  \caption{Estimated error on the principal eigenvalue and on the average of $\varphi$ with respect
    to the invariant probability measure as a function of the timestep, by Galerkin
    approximation. The eigenvalues are computed with~\eqref{eq:vpappli}.
    The first and second order averages of $\varphi$ correspond
    to~\eqref{eq:phiappli} with $j=1$ and $j=2$ respectively. The first order-TU scheme is
    computed with~\eqref{eq:TUappli}. The dashed lines show reference first and second order
    convergences.}
  \label{fig:DMC}
\end{figure}

We next consider the Monte Carlo scheme presented in Section~\ref{sec:algo}, taking
$M=5\times 10^4$ and an integration time $T=5\times 10^2$, with $\Nit=\floor{\frac{T}{\Dt}}$
for each timestep $\Dt$. We use half of the time for burn-in, 
and average in time over the second half of the simulation. Moreover, for each value of $\Dt$, 
we run~30 realizations in order to reduce the variance of the estimator and to estimate
error bars on the Monte Carlo estimates (not displayed on the pictures).
The choice of the function $\chidt$ depends on 
the scheme through~\eqref{eq:chichoice}. We compare in Figure~\ref{fig:DMC_MC} the results of 
the Monte Carlo algorithm with the Galerkin approximation, which serves as a reference. The 
agreement is very good, up to small errors arising from the finiteness of the population and 
of the simulation time. This result was expected since, given that the integration by the Euler
scheme is exact in law in this case, the Monte Carlo method must match exactly the Galerkin 
approximation provided $N$, $N_{\rm iter}$ and~$M$ are all sufficiently large.

\begin{figure}[h]
  \begin{subfigure}[h]{0.4\textwidth}  
    \includegraphics[width=\textwidth,angle=270]{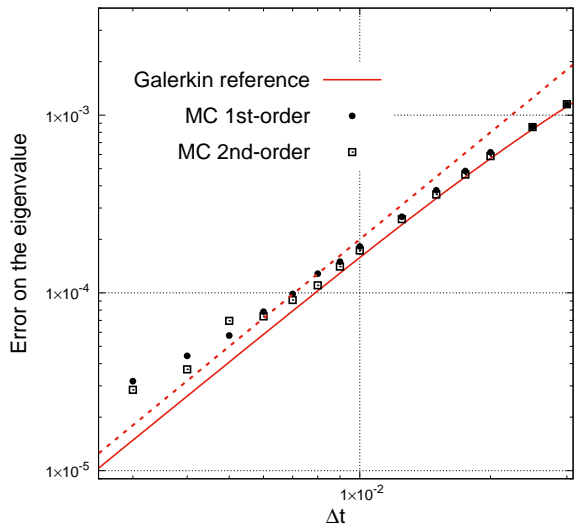}
    \caption{Convergence of the eigenvalue.}
  \end{subfigure}
  \qquad  \hspace{0.5cm}
  \begin{subfigure}[h]{0.4\textwidth}
    \includegraphics[width=\textwidth,angle=270]{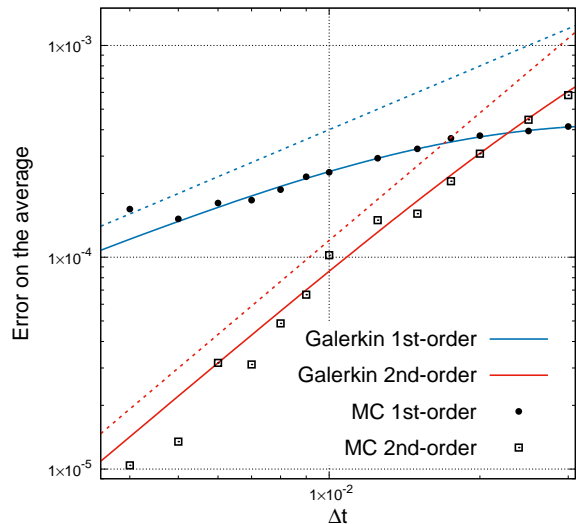}
    \caption{Convergence of the average of $\varphi$. }
  \end{subfigure}
  \caption{Estimation of the error for the principal eigenvalue and the
    average of $\varphi$ with respect
    to the invariant probability measure as a function of the timestep, by Monte Carlo simulation
    (with comparison to Galerkin, see Figure~\ref{fig:DMC}). The 
    eigenvalues and the averages of $\varphi$ are computed with~\eqref{eq:lambdaMC} 
    and~\eqref{eq:phiMC} respectively. The first and second order schemes are relative to
    the choice of the weight $\chidt$ in~\eqref{eq:chichoice}. The dashed lines show
    reference first and second convergences. For very small values of the error on the eigenvalue,
    we observe the bias due to the finite size of the population.}
  \label{fig:DMC_MC}
\end{figure}

\paragraph{Situation with a strong potential.}
We next show an application with a non-zero drift by setting $V(q)=\cos(2\pi q)$ and
$\gamma =1$. Let us recall that this dynamics is non-reversible since a constant function
is not the gradient of a smooth periodic potential. The other parameters are left unchanged.
Concerning the Galerkin approximation, we consider the two schemes described in 
Section~\ref{sec:Galerkin}, and characterized by the matrices defined in~\eqref{eq:matrices}.
For these schemes, the eigenvalues are the same and converge at second order (so we only
consider one scheme), while the averages of $\varphi$ converge at first and second
order respectively.

For the Monte Carlo algorithm described in Section~\ref{sec:algo}, we consider the three
following schemes: 
\begin{itemize}
\item $\Qdt$ is discretized with the Euler scheme~\eqref{eq:EM}, and $\chidt(q,q')=\Dt W(q)$
is chosen as the left point integration; in this case, the eigenvalue and the average of $\varphi$ 
converge at order one, so the scheme is referred to as \emph{first order}.
\item $\Qdt$ is discretized with the second order scheme~\eqref{eq:modified}, and we
set $\chidt(q,q')=\Dt W(q)$; in this case, the eigenvalue converges at second order whereas the average of
$\varphi$ converges
at first order only, so the scheme is referred to as \emph{hybrid scheme}.
\item $\Qdt$ is discretized with the second order scheme~\eqref{eq:modified}, and we
set $\chidt(q,q')=\Dt \left(W(q)/2 + W(q')/2\right)$, which corresponds to a trapezoidal rule
for the time integral; in this case, both the eigenvalue and the average of $\varphi$
converge at order two, so we refer to this scheme as \emph{second order}.
\end{itemize}

We present the numerical results obtained with the various schemes we consider in Figures~\ref{fig:strong_eig} (eigenvalues) and~\ref{fig:phi_strong_potential} (averages of~$\varphi$): 
\begin{itemize}
\item Concerning the eigenvalues computed with the Monte Carlo method,
we indeed observe first order convergence for the first order scheme, and second order convergence
for the hybrid and second order schemes. In particular, the results of the hybrid and the second
order scheme are exactly the same.
The Galerkin method also converges at second order, but with 
a much smaller prefactor. This is due to the fact that in this case most of the error is due to
the discretization of the dynamics rather than the discretization of the time integral. 
\item Concerning the average of $\varphi$, the first order scheme converges at order one, while
the hybrid and second order scheme converge at order two. We would have expected the hybrid
scheme to converge at first order but, once again, this is due to the fact that most of the error is due
to the discretization of the dynamics, and not to the time integral -- as shown by the
results of the Galerkin method, which amounts to observing the error due to the discretization
of the time integral only. We indeed observe first and second order convergence for the Galerkin approximation,
but we see that the error is orders of magnitude smaller than the one of the Monte Carlo
approximation. This explains why the Monte Carlo hybrid and second order schemes seem to provide the same results.
\end{itemize}

\begin{figure}[h]
  \centering
  \includegraphics[width=0.5\textwidth,angle=270]{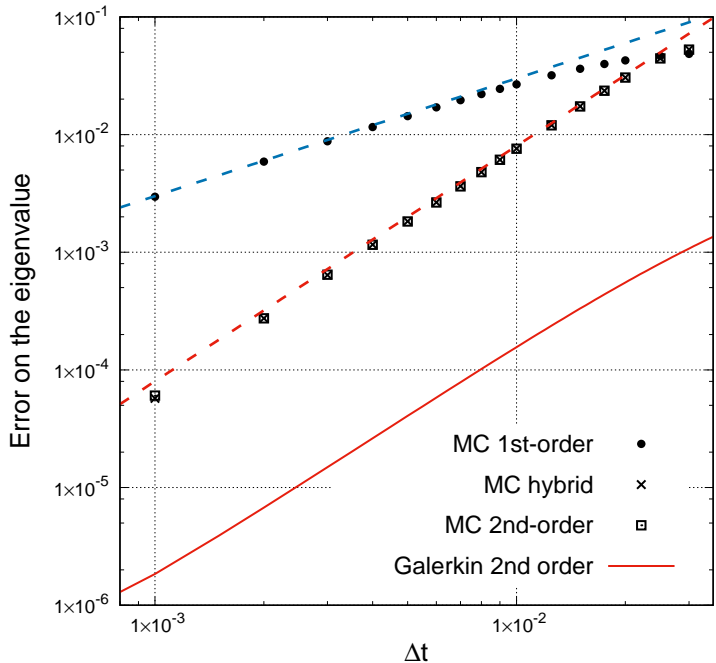}
  \caption{\label{fig:strong_eig} Estimation of the error on the principal eigenvalue as a
    function of the timestep, by Monte Carlo simulation and Galerkin approximation,
    for $V(q)=\cos(2\pi q)$. The Monte Carlo estimates of the  eigenvalues are computed
    with~\eqref{eq:lambdaMC}, while the Galerkin approximations of the eigenvalues 
    are obtained with~\eqref{eq:vpappli} for $j=1$ and $j=2$. The dashed lines show
  reference linear and quadratic convergences to zero.}
\end{figure}

\begin{figure}[h]
  \begin{subfigure}[h]{0.4\textwidth}  
    \includegraphics[width=\textwidth,angle=270]{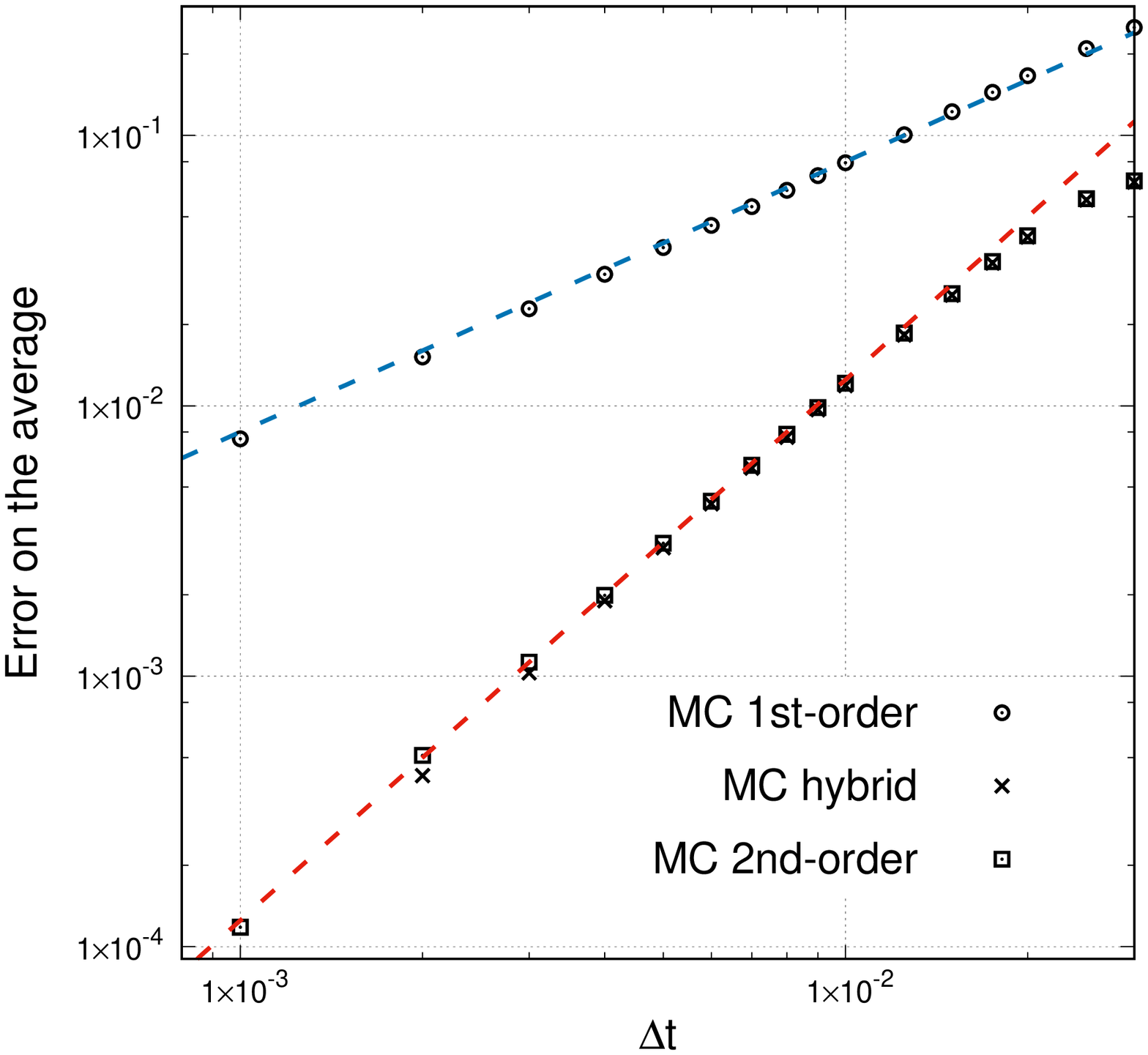}
    \caption{Convergence of the averages.}
  \end{subfigure}
  \qquad  \hspace{0.5cm}
  \begin{subfigure}[h]{0.4\textwidth}
    \includegraphics[width=\textwidth,angle=270]{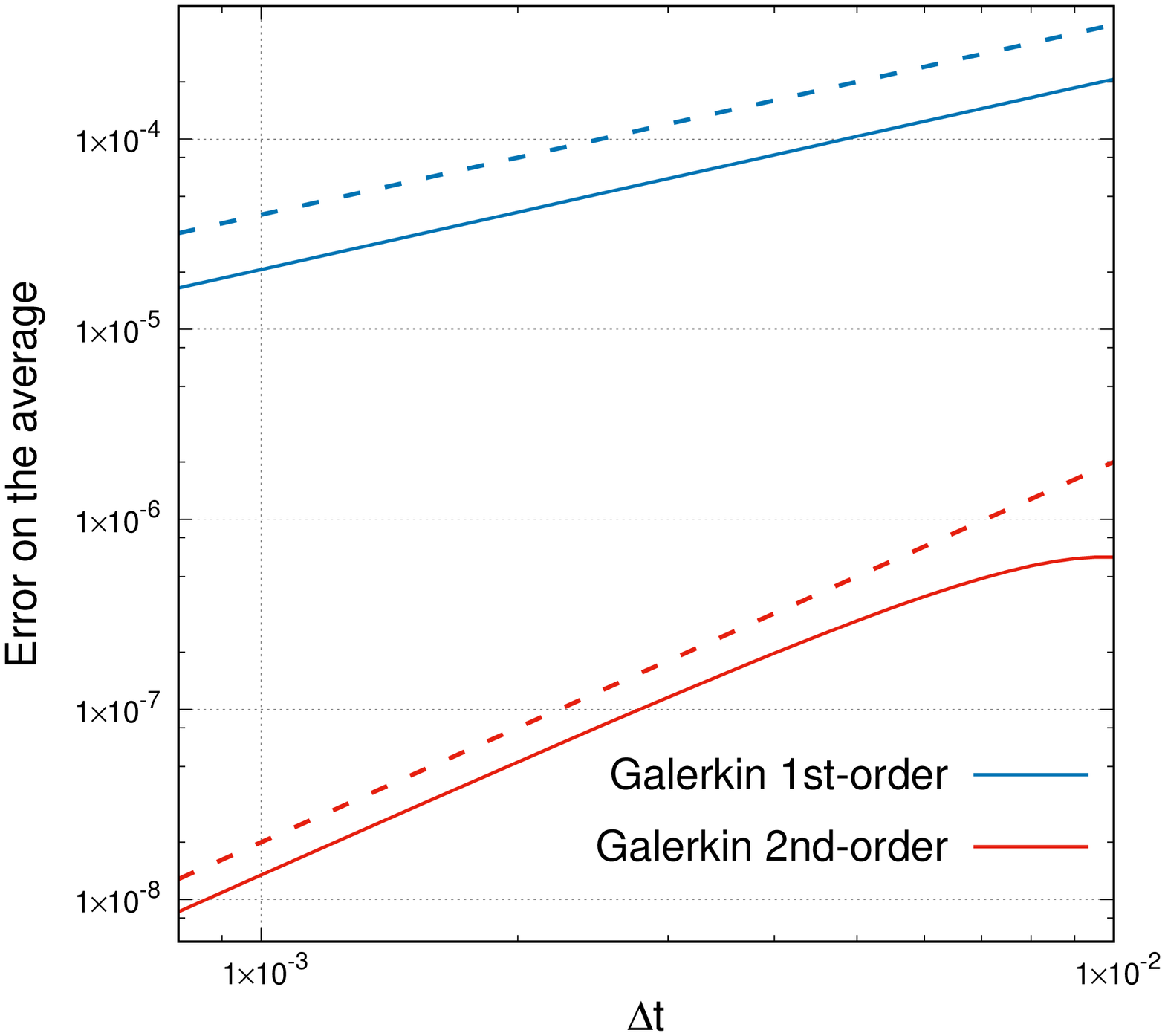}
    \caption{Zoom on the results obtained with Galerkin.}
  \end{subfigure}
  \caption{\label{fig:phi_strong_potential} Estimation of the error on the average of $\varphi$
    with respect to the invariant probability measure as a function of the timestep, by Monte
    Carlo simulation and Galerkin
    approximation, for $V(q)=\cos(2\pi q)$. The Monte Carlo estimates of 
    the averages of $\varphi$ are computed with~\eqref{eq:phiMC}, while the Galerkin approximation
    is obtained with~\eqref{eq:phiappli} for $j=1$ and $j=2$. The dashed lines show
    reference linear and quadratic convergences to zero.}
\end{figure}

\paragraph{Situation with a weak potential.}
In order to obtain a better trade-off between the error due to the discretization of the
dynamics and of the time integral, we run simulations with the same parameters as in the
previous situation but with a smaller potential energy $V(q)=0.02 \cos(2\pi q)$. The
results are the following:
\begin{itemize}
\item All the eigenvalues now seem to converge
at second order (see Figure~\ref{fig:phi_weak_potential}~(a)). This is due to the fact that the error due to the discretization of the
dynamics is very small, and that the discretization of the time integral, which gives
the dominant error term, always leads to an effective second order convergence. 
\item The behaviour of the average of $\varphi$ is more interesting (see Figure~\ref{fig:phi_weak_potential}~(b)). The Galerkin first and second order schemes provide first and second order
convergence respectively. The hybrid scheme exhibits a first
order convergence, that matches the Galerkin first order scheme for small timesteps. This result
can be expected since the two schemes match at order one. The first order scheme also converges
at first order but with a larger prefactor, which is due to the discretization of the 
dynamics. On the other hand, the second order Monte Carlo scheme converges at second order,
like the Galerkin second order scheme.
\end{itemize}

\begin{figure}[h]
  \begin{subfigure}[h]{0.4\textwidth}  
    \includegraphics[width=\textwidth,angle=270]{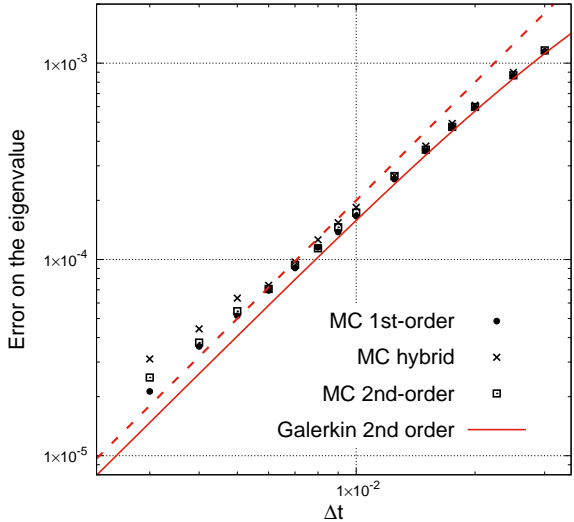}
        \caption{Convergence of the eigenvalue.}
  \end{subfigure}
  \qquad  \hspace{0.5cm}
  \begin{subfigure}[h]{0.4\textwidth}
    \includegraphics[width=\textwidth,angle=270]{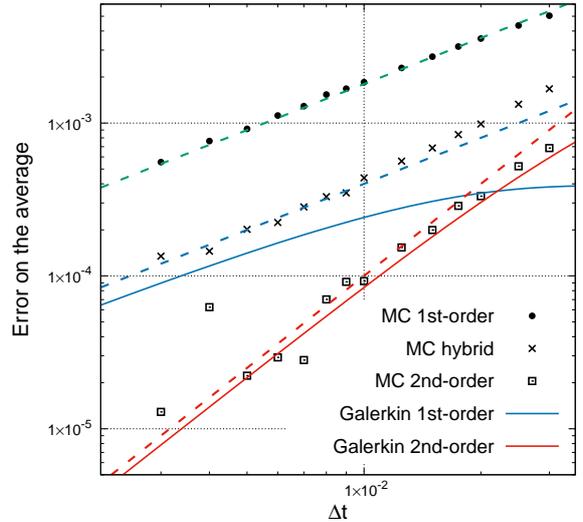}
    \caption{Convergence of the average.}
  \end{subfigure}
  \caption{\label{fig:phi_weak_potential}
    Estimation of the error on the principal eigenvalue and the
    average of $\varphi$ with respect to the invariant probability measure as a function of the
    timestep, for $V(q)=0.02 \cos(2\pi q)$. The Monte Carlo estimates of
    the averages of $\varphi$ are computed with~\eqref{eq:phiMC}, while the Galerkin approximations
    are obtained with~\eqref{eq:phiappli} for $j=1$ and $j=2$. The dashed lines show reference
    first and second order convergences. For very small values of the error on the eigenvalue,
    we observe the bias due to the finite size of the population. We also observe that the
  error on the average of~$\varphi$ becomes noisier below~$10^{-4}$.}
\end{figure}

\paragraph{Conclusion.}
The numerical applications we presented show the validity of our analysis on a simple test case.
However, we observe numerically that the prefactor of the leading error term depends on the 
choice of parameters. This has the consequence that some schemes may effectively seem to
exhibit an improved  order of convergence than expected, while they actually have a small
prefactor at leading order, depending on the discretization at hand. This observation
also motivates the study of a one-dimensional model: not only can the Galerkin discretization
be made sufficiently accurate by considering a very large number of basis functions, but
we can also run sufficiently long Monte Carlo simulations in order for the statistical
error to be negligible compared to the bias arising from the time step discretization.
Although the order of convergence would be harder to
observe for higher dimensional systems, the framework is still applicable and 
we refer to~\cite{lim2017fast} and references therein for examples in high dimension.

\section{Possible extensions}
\label{sec:extensions}

The analysis and simulations we performed in this work were done for SDEs with a non-degenerate 
noise on a torus. We however believe that most of our results could be extended to more general 
settings. The first issue is to study the long time behaviour of Feynman--Kac
dynamics and their discretizations when considering unbounded configuration spaces and/or
SDEs with degenerate noise such as inertial Langevin dynamics.
We recently addressed this problem in~\cite{ferre2018more} by using weighted function space in
the spirit of~\cite{hairer2011yet}. This provides criteria in terms of
growth conditions on~$b$ and~$W$ for extending Proposition~\ref{theo:continuousconvergence} and
Theorem~\ref{theo:feynmankacdiscr}, for instance, and H\"ormander-type conditions for dynamics
with degenerate noise.


Since the long time behaviour of unbounded dynamics has been studied, we would like to 
extend the error estimates on the ergodic properties of the dynamics
presented in Section~\ref{sec:numana} in this unbounded setting.
However, in the functional framework of~\cite{ferre2018more}, the stability property
(Assumption~\ref{as:stability}), which
is crucial for our analysis to hold, should be rephrased as the invariance
of a functional space (containing unbounded functions) under the action of
$\hw^{-1}(\Lc^* + W - \lambda)^{-1}(\hw \cdot)$. In the case $W=0$, which does not involve 
the eigenvector $\hw$ (since $h_0 =\mathds{1}$), this is already a quite technical result to 
obtain (see~\cite{kopec2014weak,kopec2015weak}). Here, the presence of the eigenvector $\hw$ 
adds a significant difficulty, which leaves the situation open.

Finally, in the context of large deviations, one is often interested in computing the rate
function, which is the Fenchel transform of the eigenvalue~$\lambda$ associated to a particular
function~$W$, see~\cite{dembo2010large}. It is an interesting and non-trivial problem to
transpose our error estimates on~$\lambda$ to error estimates on the rate function.


\section{Proofs}
\label{sec:proofs}

\subsection{Proof of the results of Section~\ref{sec:continuous}}
\label{sec:proof_continuous}

Let us first give a result which shows that it suffices to prove Proposition~\ref{theo:continuousconvergence} for probability measures which admit a positive and bounded density with respect to the Lebesgue measure. This results relies on the regularizing properties of the underlying diffusion.

\begin{lemma}
\label{lem:P_alpha}
For any $\alpha>0$, denote by $\mathcal{P}_{\alpha}(\X)$ the subspace of probability measures which admit a smooth density with respect to the Lebesgue measure, and whose density is bounded below by $\alpha>0$ and bounded above by $1/\alpha$. Then there exists $\alpha_* > 0$ such that $\Phi_1(\mu) \in \mathcal{P}_{\alpha_*}(\X)$ for any $\mu \in \PX$.
\end{lemma}

\begin{proof}
Note that, for any $\varphi\in\S$,
\[
\left(P_t^W \varphi\right)(q) = \intX p_t^W(q,q') \varphi(q') \, dq',
\]
where $p_t^W$ is the integral kernel of the semigroup $\e^{t (\Lc+W)}$. By parabolic regularity
(see for instance~\cite{evans2010partial}), the integral kernel is smooth for any $t>0$. It is
also positive when $W=0$ by the irreducibility properties of the underlying non-degenerate
diffusion: there exists $\eta > 0$ such that (setting $t=1$)
\[
\forall \, (q,q') \in \X^2, \qquad \eta \leq p_1^0(q,q') \leq \frac1\eta.
\]
Given that $W$ is bounded, a similar property holds for $p_t^W$: there exists $\alpha > 0$ such that 
\[
\forall\, (q,q') \in \X^2, \qquad \sqrt{\alpha} \leq p_1^W(q,q') \leq \frac{1}{\sqrt{\alpha}}.
\]
Since, for any bounded measurable function $\varphi$,
\[
\Phi_1(\mu)(\varphi)= \frac{\mu(P_1^W\varphi)}{\mu(P_1^W\mathds{1})} = 
\frac{1}{\mu(P_1^W\mathds{1})} \int_\X \int_\X \varphi(q') p_1^W(q,q') \, \mu(dq)dq',
\]
it follows that $\Phi_1(\mu)$ has a smooth density with respect to the Lebesgue measure, denoted by $F_{1,\mu}$:
\[
F_{1,\mu}(q) = \frac{1}{\mu(P_1^W\mathds{1})}  \int_\X p_1^W(q',q) \, \mu(dq').
\]
Moreover, since $\mu(P_1^W\mathds{1}) \geq \sqrt{\alpha}$, it holds
\[
\alpha \leq F_{1,\mu} \leq \frac1\alpha,
\]
which gives the claimed result. 
\end{proof}

We can now provide the proof of Proposition~\ref{theo:continuousconvergence}.

\begin{proof}[Proof of Proposition~\ref{theo:continuousconvergence}]
In view of the semigroup property $\Phi_t(\mu)=\Phi_{t-1}(\Phi_1(\mu))$ when $t \geq 1$, it 
is sufficient by Lemma~\ref{lem:P_alpha} to prove the result for measures 
$\mu \in \mathcal{P}_\alpha(\X)$, where $\alpha >0$. The proof is conducted in two steps: we 
first prove a convergence result for the linear semigroup $P_t^{W-\lambda}$ in $L^2(\nu)$ and any times $t \geq 0$, and then rely on the fact that any probability measure in $\mathcal{P}_\alpha(\X)$ is equivalent to~$\nu$ to obtain~\eqref{eq:feynmankaccv}.

Introduce the projector (different from the one defined in~\eqref{eq:proj})
\[
\widehat{\Pi}_W \varphi = \varphi - \hat{h}_W \frac{\big\langle \varphi, h_W \big\rangle_{L^2(\nu)}}
{\left\langle \hat{h}_W, h_W \right\rangle_{L^2(\nu)}}
 = \varphi - \hat{h}_W \frac{\int_\X \varphi \, d\nu_W}
{\int_\X \hat{h}_W \, d\nu_W}. 
\]
A simple computation shows that $\widehat{\Pi}_W$ commutes with $\Lc+W$ and $P_t^{W-\lambda}$. It is easily seen that the spectrum of the operator $\widehat{\Pi}_W (\Lc+W-\lambda)\widehat{\Pi}_W$ is 
\[
\sigma(\Lc+W-\lambda) \setminus \{ 0 \} \subset \Big\{z \in \mathbb{C}, \, \mathrm{Re}(z) \leq -\delta_W \Big\},
\]
and that the associated semigroup satisfies $\e^{t \widehat{\Pi}_W (\Lc+W-\lambda)\widehat{\Pi}_W} = P_t^{W-\lambda} \widehat{\Pi}_W$. By the Hille--Yosida theorem (see for instance~\cite{Pazy,DautrayLions5}), there exists therefore a constant $C>0$ such that, for any $\varphi\in L^2(\nu)$, 
\begin{equation}
\label{eq:HY}
\forall\, t \geq 0, \qquad \left\| P_t^{W-\lambda} \widehat{\Pi}_W \varphi \right\|_{L^2(\nu)} \leq C\, \e^{-\delta_W t} \| \varphi \|_{L^2(\nu)}.
\end{equation}

We now show that~\eqref{eq:HY} implies the convergence result~\eqref{eq:feynmankaccv} for the class of probability measures $\mathcal{P}_{\alpha}(\X)$. For a given $\mu\in\mathcal{P}_{\alpha}(\X)$ and $\varphi\in L^2(\nu)$, 
\[
\begin{aligned}
\mu\left[ \left|P_t^{W-\lambda}\widehat{\Pi}_W \varphi\right|\right] 
&= \intX \left| P_t^{W-\lambda}\widehat{\Pi}_W \varphi \right| \, d\mu \leq \frac{1}{\alpha} \intX \left| P_t^{W-\lambda}\widehat{\Pi}_W \varphi
 \right|  \, \frac{d \nu}{\inf_{\X}\nu}
\\ & \leq \frac{1}{\alpha \inf_{\X}\nu} \left\| P_t^{W-\lambda}\widehat{\Pi}_W \varphi \right\|_{L^2(\nu)} \leq \frac{C }{\alpha \inf_{\X}\nu} \e^{-\delta_W t} \left\| \varphi\right\|_{L^2(\nu)},
\end{aligned}
\]
where we used a Cauchy-Schwarz inequality on $L^2(\nu)$ to go from the first to the
second line. The latter computation shows that, for any $\mu\in \mathcal{P}_{\alpha}(\X)$ 
and $\varphi\in L^2(\nu)$, there are functions $a_t,b_t$ for which 
\[
\mu\left(P_t^{W-\lambda}\varphi\right) = \frac{\dps \intX \hh\,d\mu}{\dps \int_\X \hat{h}_W \, d\nu_W} \intX \varphi\, d\nuw +a_t, 
\qquad 
\mu\left(P_t^{W-\lambda}\mathds{1}\right)= \frac{\intX \hh\,d\mu}{\dps \int_\X \hat{h}_W \, d\nu_W} +b_t,
\]
with $|a_t|\leq K \|\varphi \|_{L^2(\nu)} \e^{-\delta_W t}$  
and $| b_t | \leq K \e^{-\delta_W t}$ for some constant $K>0$ independent of $\mu$ and $\varphi$. 
Moreover, there exists $\varepsilon >0$ such that $\varepsilon \leq \hh\leq 1/\varepsilon$.
Note also that $|b_t| \leq \varepsilon^2$ for $t \geq \ln(K/\varepsilon^2)/\delta_W$ and that
\[
\frac{\intX \hh \,d\mu}{\int_\X \hat{h}_W \, d\nu_W} \geq \varepsilon^2.
\]
Since 
\[
\Phi_t^W(\mu)(\varphi) = \frac{\mu\left(P_t^{W}\varphi\right)}{\mu\left(P_t^{W}\mathds{1}\right)} = \frac{\mu\left(P_t^{W-\lambda}\varphi\right)}{\mu\left(P_t^{W-\lambda}\mathds{1}\right)},
\]
it follows that, for $t \geq \ln(2K/\varepsilon)/\delta_W$,
\[
\begin{aligned}
\left| \Phi_t^W(\mu)(\varphi) - \intX \varphi\,d\nuw \right| & =
\left| \frac{ \intX \hh \,d\mu \intX \varphi \, d\nuw + a_t \int_\X \hat{h}_W \, d\nu_W}{\intX \hh \,d\mu+ b_t \int_\X \hat{h}_W \, d\nu_W}  - \intX \varphi\,d\nuw \right|
= \left| \frac{\left(a_t - b_t \intX \varphi \, d\nuw\right)\int_\X \hat{h}_W \, d\nu_W}{\intX \hh \,d\mu+b_t\int_\X \hat{h}_W \, d\nu_W} \right| \\
&
\leq \frac{1}
     { \frac{\intX \hh \,d\mu }{\int_\X \hat{h}_W \, d\nu_W} -|b_t| }
  \left( |a_t| + |b_t| \left|\intX  \varphi \, d\nuw\right|
  \right)\\
& \leq \frac{K}{2\varepsilon^2} \left(\left\|\varphi \right\|_{L^2(\nu)} + \left|\intX \varphi \, d\nuw\right| \right) \e^{-\delta_W t}.
\end{aligned}
\]
The inequality 
\[
\left|\intX \varphi \, d\nuw\right| = \left|\intX \varphi h_W \, d\nu\right| \leq \|h_W\|_{L^2(\nu)} \|\varphi \|_{L^2(\nu)}
\]
allows to obtain the desired conclusion.
\end{proof}

Let us conclude this section with the proof of Proposition~\ref{prop:inversibility}.

\begin{proof}[Proof of Proposition~\ref{prop:inversibility}]
The exponential convergence result~\eqref{eq:HY} implies that the operator $\widehat{\Pi}_W(\Lc + W - \lambda)\widehat{\Pi}_W$ is invertible on $\mathrm{Ran}(\widehat{\Pi}_W) = L_W^2(\nu)$ 
with inverse given by
\[
\Big( \widehat{\Pi}_W(\Lc + W - \lambda)\widehat{\Pi}_W \Big)^{-1} = - \int_0^{+\infty} P_t^{W - \lambda}\widehat{\Pi}_W \, dt.
\]
The solution to $(\Lc + W - \lambda)u=g$ with $g\in L_W^2(\nu)$ then admits a unique solution in $L_W^2(\nu)$. By elliptic regularity, $u \in \Sw$ when $g \in \Sw$. The result for $\Lc^* + W - \lambda$ can be obtained by a similar reasoning.
\end{proof}

Note that, alternatively, it would have been possible to resort to the Fredholm alternative to prove Proposition~\ref{prop:inversibility}.

\subsection{Proof of Theorem~\ref{theo:feynmankacdiscr}}
\label{sec:feynmankacdiscr}

Theorem~\ref{theo:feynmankacdiscr} is a rewriting of~\cite[Corollary 2.5]{del2001stability},
which is stated in the context of a finite state space. In order for the paper to be self-contained,
we prove Theorem~\ref{theo:feynmankacdiscr} in our setting of continuous
but compact state space, and in the simplified case of a time-homogeneous Markov chain,
adapting the arguments of~\cite{del2001stability}.
The idea is to prove some contraction property using the Dobrushin coefficient defined
in Appendix~\ref{sec:dobrushin} and the reformulation~\eqref{eq:mupath} below of the semigroup.
We work on the space of  probability measures $\PX$ endowed with the total variation distance.

Define the weights
\[
g_n = (\Qdt^W)^{n}\mathds{1},
\]
and the Markov operator $S_n$ as
\begin{equation}
\label{eq:Sn}
(S_n\varphi)(q) = \frac{ \Qdt^W(g_n \varphi)(q)}{(\Qdt^W g_n)(q)}.
\end{equation}
The dynamics~\eqref{eq:generaldiscr} can then be rephrased as 
\begin{equation}
\label{eq:mupath}
\mu_n(\varphi) = \frac{ \mu\left(g_n (K_n\varphi)   \right) }{\mu(g_n)}, \quad
K_{n+1}=S_n K_n, \quad K_0=\Id.
\end{equation}
This equality can be proved by induction. The result is clear for $n=0$. For $n=1$, we have
that (with $\mu_0 = \mu$)
\[
\frac{ \mu\left(g_1 (K_1\varphi)   \right) }{\mu(g_1)}
= \frac{ \mu\left(\Qdt^W\mathds{1} \, S_0\varphi   \right) }{\mu\left(\Qdt^W\mathds{1}\right)}
= \frac{ \mu\left( \Qdt^W\varphi \right) }{\mu\left(\Qdt^W \mathds{1}\right)}=\mu_1(\varphi).
\]
Assuming that $\mu_n$ satisfies~\eqref{eq:mupath} at rank $n$, using~\eqref{eq:generaldiscr} and
recalling the definition~\eqref{eq:def_mu_Q}, 
\[
\mu_{n+1}(\varphi) = \frac{\mu\left( (\Qdt^W)^{n+1}(\varphi) \right)}
{\mu\left( (\Qdt^W)^{n+1}\mathds{1} \right)}
= \dps \frac{\left(\mu\Qdt^W\right)\left( (\Qdt^W)^{n}(\varphi) \right)}
{\left(\mu\Qdt^W\right)\left( (\Qdt^W)^{n}\mathds{1} \right)}
 =  \Phi_{\Dt,n}^W\left(\mu\Qdt^W\right)(\varphi),
\]
so that, using the recursion hypothesis and $\Qdt^W g_n = g_{n+1}$, it follows
\[
\begin{aligned}\dps
\mu_{n+1}(\varphi) = & \dps \frac{\dps \left(\mu \Qdt^W\right)\left(g_n (K_n\varphi)   
\right) }{(\mu \Qdt^W)(g_n)}
  =   \dps  \frac{ \mu\left( \Qdt^W(g_n (K_n\varphi))   \right) }{\mu(\Qdt^Wg_n)}
\\ =  & \dps  \frac{ \mu\left( \Qdt^W(g_n) S_n(K_n\varphi)\right)  }{\mu(\Qdt^Wg_n)}
  =  \dps  \frac{\dps \mu\left( g_{n+1} (K_{n+1}\varphi)   \right) }
{\dps \mu\left(g_{n+1}\right)},
\end{aligned}
\]
which concludes the recurrence.

We next introduce the familly of operators $T_n:\PX\to\PX$ defined by:
\[
\label{eq:Tn}
\forall\,\mu\in\PX, \quad \forall\, \varphi\in\S, \qquad
(\mu T_n)(\varphi)=\frac{\mu\left( g_n \varphi \right) }{\mu(g_n)},
\]
so that from~\eqref{eq:mupath} we have $\mu_n =\Phi_{\Dt,n}^W(\mu)=  \mu T_n K_n$. 
Using the definitions of Appendix~\ref{sec:dobrushin}, we obtain, for two initial
measures $\mu,\nu\in\PX$,
\[
\|\mu_n - \nu_n \|\TV =\|\mu T_n K_n-\nu T_n K_n \|\TV \leq \vertiii{K_n}
\,  \| \mu T_n - \nu T_n \|\TV. 
\]
Given that $T_n:\PX\to\PX$, we can bound $\|\mu T_n - \nu T_n \|\TV$ by $2$. The next step
consists in studying the contraction induced by the operator $K_n=S_{n-1} S_{n-2} \hdots S_1$,
where $S_k$ is defined in~\eqref{eq:mupath}. We have
\[
\vertiii{K_n} \leq \prod_{k=0}^{n-1}\vertiii{S_k},
\]
so that, using the relationship~\eqref{eq:contractiondobrushin},
\[
\|\mu_n - \nu_n \|\TV \leq 2 \prod_{k=0}^{n} \left( 1 - \alpha(S_k) \right).
\]
The last step consists in using Assumption~\ref{as:minorization} in order to obtain a 
lower bound on $\alpha(S_k)$ independent of~$k$. First, for all $q\in\X$ and
$A\subset\X$,
\[
S_k(q,A)= \frac{\Qdt^W(g_k \mathds{1}_A )(q)}{\Qdt^W(g_k)(q)} \geq \varepsilon^2 
\frac{\eta(g_k \mathds{1}_A )}{\eta(g_k)}.
\]
Then, it follows from definition~\eqref{eq:dobrushin} that
\[
\label{eq:lowerbound}
\alpha(S_k) = \inf_{\substack{ q,q'\in\X \\ \{A_i\}_{1 \leq i \leq m} \subset\X }} \! \left\{\sum_{i=1}^m \min \left(  S_k(q,A_i), S_k(q',A_i)\right) \right\}
\geq \varepsilon^2 \underset{ \{A_i\}_{1 \leq i \leq m} \subset\X}{\inf}\left\{\frac{\eta\left(g_k \sum_{i=1}^m \mathds{1}_{A_i} \right) } {\eta(g_k)} \right\} = \varepsilon^2, 
\]
since the infimum is taken over partitions $(A_i)_{i=1}^m$ of $\X$.
As a result, we obtain that, for all measures $\mu,\nu\in\PX$,
\begin{equation}
\label{eq:contractionTV}
\left\| \Phi_{\Dt,n}^W(\mu) - \Phi_{\Dt,n}^W(\nu) \right\|\TV \leq 2 \left(1-\varepsilon^2
\right)^n.
\end{equation}
Setting $\nu=\mu_m$ for $m\in\N$ and using the semigroup property, we get
\begin{equation}
\label{eq:exp_Cauchy_like}
\left\| \Phi_{\Dt,n}^W(\mu) - \Phi_{\Dt,n+m}^W(\mu) \right\|\TV \leq 2 \left(1-\varepsilon^2\right)^n,
\end{equation}
so that $(\mu_n)_{n\geq 1}$ is a Cauchy sequence in $\PX$. By completeness of $\PX$ for the
total variation norm, we can conclude that, for any initial measure $\mu$, 
there exists $\muinfty$ such that $\mu_n\to\muinfty$ in total variation norm. Then using the 
one step formulation of the dynamics~\eqref{eq:etadynamics} and the semigroup property, we 
obtain with the choice $\nu=\K\mu$,
\[
\left\| \Phi_{\Dt,n}^W(\mu) - \K \Phi_{\Dt,n}^W(\mu) \right\|\TV \leq 
2 \left(1-\varepsilon^2\right)^n,
\]
so that, taking $n\to\infty$ and using the continuity of $\K$ on $\PX$ endowed with the total variation norm, it follows that $\muinfty=\K\muinfty$. Passing to the limit $m\to+\infty$ in~\eqref{eq:exp_Cauchy_like}, 
\[
\| \mu_n - \muinfty \|\TV \leq 2 \left(1-\varepsilon^2\right)^n.
\]
Finally, it follows from~\eqref{eq:contractionTV} that the limit $\muinfty$ does not depend on 
the initial measure $\mu$.

\subsection{Proofs related to Theorem~\ref{theo:numana}}

\subsubsection{Proof of Lemma~\ref{lem:approxstationarity}}
\label{sec:approxstationarity}

The idea is to approximate at leading order the stationary measure $\nuwdt$ as $(1+\Dt^p f)\nuw$, since we expect the invariant probability measure to be correct at order $p$. We start from the stationarity equation~\eqref{eq:approxstationary} and search for a function~$f \in \S$ and a remainder~$R_{W,\Dt}:\S\to\R$ satisfying~\eqref{eq:uniformbound} such that, for all $\phi \in \S$,
\begin{equation}
\label{eq:intermediate}
\intX (\Qdt^W \phi)(1+\Dt^p f)\, d\nuw -\left( \intX \Qdt^W\mathds{1}(1+\Dt^p f )
\,d\nuw\right) \left( \intX \phi (1 +\Dt^p f )\, d\nuw\right) = \Dt^{p+2}R_{W,\Dt}\phi.
\end{equation}
In view of the expansion~\eqref{eq:dlsemigroupgeneral} of~$\Qdt^W$ and of the invariance relation~\eqref{eq:invariance}, the first term
of the left hand side is
\[
\begin{aligned}
& \intX \left(\phi + \Dt\At_1\phi + \hdots +\Dt^p \At_p\phi +\Dt^{p+1}\At_{p+1}\phi +\Dt^{p+2} \Rm_{W,\Dt} \phi \right) \left( 1 +\Dt^p f\right) d\nuw \\
& = \left( 1 + a_1\Dt + \hdots + a_p \Dt^p \right)\intX\phi\, d\nuw +\Dt^p\intX \phi f\, d\nuw + \Dt^{p+1}\intX \left( \At_{p+1}\phi +f(\At_1\phi)\right) d\nuw +\Dt^{p+2} R_{W,\Dt}\phi,
\end{aligned}
\]
where $R_{W,\Dt}$ gathers the terms of order at least $p+2$, and is uniformly bounded in $\Dt$ for $0<\Dt\leq\Dt^*$ in the sense of~\eqref{eq:uniformbound} when $f\in\S$. On the other hand, the second term on the left hand side of~\eqref{eq:intermediate} can be written as, using again~\eqref{eq:invariance},
\[
\begin{aligned}
&  \left( 1 + \Dt a_1 + \hdots + \Dt^p a_p + \Dt^p \intX f\, d\nuw 
+\Dt^{p+1} \intX \left(\At_{p+1}\mathds{1} + f (\At_1\mathds{1})\right) \,d\nuw \right)
\intX \phi(1+\Dt^p f)\,d\nuw \\
& \qquad + \Dt^{p+2}R_{W,\Dt}\phi \\
& = \left( 1 +\Dt a_1 + \hdots + \Dt^p a_p\right) \intX \phi\, d\nuw 
+\Dt^p \left( \int\phi\, d\nuw \intX f d\nuw + \intX \phi f\, d\nuw \right) \\
& +\Dt^{p+1} \left( \intX \At_{p+1}\mathds{1}\,d\nuw \intX \phi\, d\nuw
+a_1 \intX \phi f \, d\nuw + \intX f(\At_1\mathds{1})\,d\nuw \intX \phi\, d\nuw
\right) + \Dt^{p+2}R_{W,\Dt}\phi,
\end{aligned}
\]
where $R_{W,\Dt}$ is uniformly bounded in $\Dt$ in the sense of~\eqref{eq:uniformbound} when $f\in\S$. We can now equate the different orders in powers of $\Dt$ on both sides of~\eqref{eq:intermediate} and choose $f$ such that only a remainder of order $p+2$ remains. The terms $a_k\Dt^k\intX\phi\, d\nuw$ cancel, so the first non-trivial condition to be satisfied to eliminate terms of order $\Dt^p$ reads
\[
\intX \phi f\, d\nuw =\left( \intX\phi\, d\nuw\right)\left( \intX f\, d\nuw\right) 
+ \intX \phi f\, d\nuw.
\]
This equality is satisfied for all $\phi\in\S$ if and only if (take \textit{e.g.} $\phi=f$)
\begin{equation}
\label{eq:fortho}
\intX f\, d\nuw =0.
\end{equation}
The condition arising from the equality of terms of order $\Dt^{p+1}$ is
\[
\intX \left(\At_{p+1} \phi + f (\At_1\phi) \right)\hw \, d\nu = a_1 \intX \phi f \, d\nuw
+ \left(\intX \left( (\At_{p+1}\mathds{1}) +f(\At_1\mathds{1}) \right) \hw \, d\nu\right)
\intX \phi\, d\nuw.
\]
Using that $\At_1\mathds{1}=W$ along with condition~\eqref{eq:invariance}, we have $a_1=\lambda$. In addition, taking adjoints in $L^2(\nu)$ and recalling $\At_1 = \A_1+W$, 
\[
\intX \phi \left( (\At_{p+1})^* \hw + (\A_1^* +W - \lambda)(\hw f) \right) d\nu
= \left( \intX \left( (\At_{p+1})^*\hw + (\A_1^* + W)(\hw f)\right) d\nu\right)
 \intX \phi\, d\nuw.
\]
Moreover, in view of~\eqref{eq:fortho}, one can subtract~$\left(  \lambda \intX f \hw\,d\nu \right) \left(\intX \phi\, d\nuw\right)$ from the right hand side of last equation. Finally, we obtain the following equation (with unknown $f$): for all $\phi \in \S$, 
\begin{equation}
  \label{eq:fequation}
  \intX \phi \left((\At_{p+1})^* \hw + (\A_1^* +W - \lambda)(\hw f) \right) d\nu = 
  \left(\intX (\At_{p+1})^*\hw + (\A_1^* +W - \lambda)(\hw f) \, d\nu \right) \intX \phi \, d\nuw.
\end{equation}
By Assumption~\ref{as:stability}, the operator $\A_1^* +W-\lambda$ is invertible on $\Shw$ and leaves this space invariant. We can therefore define a solution $f_0$ to the following equation:
\begin{equation}
\label{eq:f0}
\left\{
\renewcommand{\arraystretch}{2.5} 
\begin{array}{l}
\dps (\A_1^* + W  - \lambda)(\hw f_0) = \tilde{g},
\\ \dps
\tilde{g}= - (\At_{p+1})^*\hw + \hw \frac{ \intX \left((\At_{p+1})^* \hw\right)\hh d\nu}
{ \intX\hh \hw d\nu}
\in\Shw. 
\end{array}
\right.
\end{equation}
The function $h_W f_0$ is uniquely defined in~$\Shw$ by Assumption~\ref{as:stability} since 
$\tilde{g}$ has average $0$ with respect to $\nuhw$, and one can check that it is indeed solution 
of~\eqref{eq:fequation}. Since the eigenvector $\hw$ is regular with $\hw >0$, the function 
$f_0$ belongs to $\S$. However, $f_0$ is not a priori of average $0$ with respect to $\nuw$, so 
that condition~\eqref{eq:fortho} is not satisfied. We can however consider the 
function~$f_{\alpha}= f_0 + \alpha$, which is still such that~\eqref{eq:fequation} holds. 
The choice $\alpha=-\intX f_0\,d\nuw$ ensures that~\eqref{eq:fortho} is satisfied. This provides 
the solution~\eqref{eq:leading} and concludes the proof.

\subsubsection{Proof of Lemma~\ref{lem:step2}}
\label{sec:step2}

We start by considering~\eqref{eq:shortstationarity} and~\eqref{eq:shortapproxstationary} for $\varphi=\Pi_W\phi$ with $\phi\in\S$:
\begin{equation}
\label{eq:stationarities1}
\intX\left[ \left(\frac{\Qdt^W -\e^{\Dt \lambdadt}}{\Dt}\right) 
 \Pi_W \phi\right] d\nuwdt = 0,
\end{equation}
and
\begin{equation}
\label{eq:stationarities2}
 \intX \left[ \left( \frac{\Qdt^W - \e^{\Dt \lambdatdt}}{\Dt}\right) \Pi_W \phi \right]
(1 +\Dt^p f)\, d\nuw = \Dt^{p+1} R_{W,\Dt}\phi.
\end{equation}
We next stabilize the operator in $\Sw$ by another application of the projector $\Pi_W$.
First,
\begin{equation}
\label{eq:intermediate1}
\begin{aligned}
& \intX\left[ \Pi_W \left(\frac{\Qdt^W -\e^{\Dt \lambdadt}}{\Dt}\right) 
\Pi_W \phi\right] d\nuwdt  \\
& \qquad = \intX\left[ \left(\frac{\Qdt^W -\e^{\Dt \lambdadt}}{\Dt}\right) 
\Pi_W \phi\right] d\nuwdt - \intX\left[ \left(\frac{\Qdt^W -\e^{\Dt \lambdadt}}{\Dt}\right) 
\Pi_W \phi\right] d\nuw \\ 
& \qquad = - \dps \intX\left[ 
\left(\frac{\Qdt^W -\e^{\Dt \lambdadt}}{\Dt}\right) 
\Pi_W \phi\right] d\nuw,
\end{aligned}
\end{equation}
thanks to~\eqref{eq:stationarities1}. Second, since $f$ has average~0 with respect to~$\nu_W$,
\[
\begin{aligned}
& \intX \left[\Pi_W \left( \frac{\Qdt^W - \e^{\Dt \lambdatdt}}{\Dt}\right) \Pi_W \phi \right]
(1 +\Dt^p f)\, d\nuw  \\
& \qquad  = \dps \intX \left[\left( \frac{\Qdt^W - \e^{\Dt \lambdatdt}}{\Dt}\right) \Pi_W \phi \right] (1 +\Dt^p f)\, d\nuw - \left(\intX \left[ \left( \frac{\Qdt^W - \e^{\Dt \lambdatdt}}{\Dt}\right) \Pi_W \phi \right] d\nuw\right) \intX (1 + \Dt^p f)  d\nu_W \\
& \qquad = \dps \intX \left[\left( \frac{\Qdt^W - \e^{\Dt \lambdatdt}}{\Dt}\right) \Pi_W \phi \right] (1 +\Dt^p f)\, d\nuw - \intX \left[ \left( \frac{\Qdt^W - \e^{\Dt \lambdatdt}}{\Dt}\right) \Pi_W \phi \right] d\nuw. 
\end{aligned}
\]
In view of~\eqref{eq:stationarities2}, the first term of the right hand side
of the above equation is a remainder of order $\Dt^{p+1}$. Therefore,
\begin{equation}
\label{eq:intermediate2bis}
\begin{aligned}
& \intX \left[\Pi_W \left( \frac{\Qdt^W - \e^{\Dt \lambdatdt}}{\Dt}\right) \Pi_W \phi \right]
(1 +\Dt^p f)\, d\nuw  = - \intX \left[ \left( \frac{\Qdt^W - \e^{\Dt \lambdatdt}}{\Dt}\right) \Pi_W \phi \right] d\nuw 
+ \Dt^{p+1} R_{W,\Dt}\phi
\\ & \qquad = -\intX \left[ \left( \frac{\Qdt^W - \e^{\Dt \lambdadt}}{\Dt}\right) \Pi_W \phi \right] d\nuw 
+ \Dt^{p+1}  R_{W,\Dt}\phi + \left( \frac{\e^{\Dt \lambdatdt} - \e^{\Dt \lambdadt}}{\Dt}\right) \intX  \Pi_W \phi\,  d\nuw 
\\ & \qquad = -\intX \left[ \left( \frac{\Qdt^W - \e^{\Dt \lambdadt}}{\Dt}\right) \Pi_W \phi \right] d\nuw  + \Dt^{p+1} R_{W,\Dt}\phi,
\end{aligned}
\end{equation}
since $\Pi_W \phi$ has average $0$ with respect to $\nuw$. Combining~\eqref{eq:intermediate2bis}
with~\eqref{eq:intermediate1},
\[
\begin{aligned}
\intX\left[ \Pi_W \left(\frac{\Qdt^W -\e^{\Dt \lambdadt}}{\Dt}\right) 
\Pi_W \phi\right] d\nuwdt & =  \intX \left[\Pi_W \left( \frac{\Qdt^W - 
\e^{\Dt \lambdatdt}}{\Dt}\right) \Pi_W \phi \right](1 +\Dt^p f) d\nuw  \\  
& \qquad + \Dt^{p+1}R_{W,\Dt}\phi,
\end{aligned}
\]
where $R_{W,\Dt}$ satisfies~\eqref{eq:uniformbound}. This concludes the proof of the lemma.

\subsubsection{Proof of Lemma~\ref{lem:step3}}
\label{sec:step3}

The first part of the proof of Lemma~\ref{lem:step3} consists in constructing an approximate eigenvector $\hwdt$ of $\hh$ for the evolution operator $\Qdt^W$. We use to this end Assumption~\ref{as:spectralconsistency} and~\eqref{eq:invariance}, as well as the definition of the leading order correction $f$ in~\eqref{eq:leading}. More precisely, we consider $\hwdt=u_0 + \Dt\, u_1 + \hdots + \Dt^p u_p\in$ and look for functions $u_1,\hdots,u_p\in\Sw$ and $u_0 \in \S$ with $\intX u_0 \, d\nu = 1$ such that
\begin{equation}
\label{eq:eigenQ}
\Qdt^W \hwdt = \e^{\Dt \lambdatdt}\hwdt +\Dt^{p+2} r_{W,\Dt},
\end{equation}
with $\|r_{W,\Dt}\|_{\Linfty} \leq C $ for $0<\Dt\leq \Dt^*$. Recall that, by~\eqref{eq:tildelambdadt},
\[
\e^{\Dt \lambdatdt}=\intX \Qdt^W\mathds{1}(1+\Dt^p f )\, d\nuw.
\]
Expanding the left hand side of~\eqref{eq:eigenQ} using~\eqref{eq:dlsemigroupgeneral} leads to
\begin{equation}
\label{eq:lhs3}
\Qdt^W \hwdt=\sum_{k=0}^{p+1} \Dt^k \At_k \hwdt + \Dt^{p+2}\Rm_{W,\Dt}\hwdt = 
\sum_{k=0}^{p+1} \Dt^{k} \sum_{m=0}^{k} \At_m u_{k-m}
+ \Dt^{p+2}\Rm_{W,\Dt}\hwdt, 
\end{equation}
with the convention $\At_0=\Id$ and $u_{p+1}=0$. The right hand side of~\eqref{eq:eigenQ} can be expanded as
\begin{equation}
\label{eq:rhs3}
\begin{aligned}
& \e^{\Dt \lambdatdt}\hwdt \\
& = \left[\intX \left(1 + \Dt \At_1\mathds{1} + \hdots + \Dt^{p+1} \At_{p+1}\mathds{1} + \Dt^{p+2}\Rm_{W,\Dt}\mathds{1}\right)(1+\Dt^p f)\, d\nuw \right]  \left( u_0 + \Dt\, u_1 + \hdots + \Dt^p u_p\right) \\
& = \dps \left[ 1 + \Dt\intX \At_1\mathds{1}\,d\nuw + \hdots + \Dt^p \intX \At_p\mathds{1}\,d\nuw + \Dt^{p+1}  \intX\left( \At_{p+1}\mathds{1} + f \At_1 \mathds{1}\right)d\nuw   + \Dt^{p+2}r_{W,\Dt} \right] \\ 
& \dps \qquad \qquad \times  \left( u_0 + \Dt\, u_1 + \hdots  + \Dt^p u_p\right) \\ 
& = \sum_{k=0}^{p+1} \Dt^{k} \sum_{m=0}^{k}  \lambda_m u_{k-m} + \Dt^{p+2}r_{W,\Dt},
\end{aligned}
\end{equation}
where we introduced $\lambda_0 = 1$,
\begin{equation}
\label{eq:lambdam1}
\forall\, m\in \{1,\hdots,p\}, \qquad \lambda_m=\intX \At_m \mathds{1}\, d\nuw,
\end{equation}
and $\lambda_{p+1}$ is defined in~\eqref{eq:lambdap1}:
\[
\lambda_{p+1} = \intX \At_{p+1} \mathds{1}\,d\nuw + \intX Wf\, d\nuw.
\]
We see from~\eqref{eq:invariance} that $\lambda_m=a_m$ for $m\in\{1,\hdots,p\}$, with in particular $\lambda_1=\intX W d\nuw=\lambda$.

We now build the functions $u_m$ by induction. Let us show the first steps of the recurrence, before proceeding to the general argument. Plugging~\eqref{eq:lhs3} and~\eqref{eq:rhs3} in~\eqref{eq:eigenQ}, the equality of terms of order~1 leads to the trivial equality $u_0=u_0$. Equating terms of order~$\Dt$ gives
\[
\At_1 u_0 + \At_0 u_1 = \lambda_1 u_0 + \lambda_0 u_1,
\]
so that, using $\At_0=\Id$, $\lambda_0=1$, $\At_1=\A_1+W$ and $\lambda_1=\lambda$,
\[
(\A_1 + W)u_0=\lambda u_0.
\]
In view of Assumption~\ref{as:spectralconsistency}, we can conclude that $u_0=\hh$. The identification of terms of order $\Dt^2$ in~\eqref{eq:lhs3}-\eqref{eq:rhs3} leads to
\[
\At_2 u_0 + \At_1 u_1 + \At_0 u_2  = \lambda_2 u_0 + \lambda_1 u_1 + \lambda_0 u_2,
\]
which can be rewritten as
\begin{equation}
\label{eq:defu1}
(\A_1 + W - \lambda)u_1 = g_{1,0}, \qquad g_{1,0}= - \At_2 \hh + \lambda_2\hh,
\end{equation}
where the expression of $\lambda_2$ is given by~\eqref{eq:lambdam1} when $p \geq 2$ and 
by~\eqref{eq:lambdap1} when $p=1$. In order to prove that~\eqref{eq:defu1} is well-posed, it 
is sufficient to show that $g_{1,0}$ belongs to $\Sw$. We show in fact in the sequel that each 
function~$u_k$ is solution to a Poisson equation similar to~\eqref{eq:defu1} with a right-hand side that always belongs to~$\Sw$. 

Let us now present the inductive construction to any order, until the terminal case $k=p$, showing in particular the well-posedness of the equations defining each mode $u_k$. This construction is reminiscent of techniques used to build the expansion of the invariant probability measure in $\Dt$ in related works, in particular~\cite{debussche2012weak}. Suppose that we have built functions $u_0, \dots, u_{k} \in \Sw$ for some $k\geq 1$. Inserting again~\eqref{eq:lhs3} and~\eqref{eq:rhs3} into~\eqref{eq:eigenQ} and equating terms of order $\Dt^{k+1}$ then leads to
\begin{equation}
\label{eq:ukrec}
 \sum_{m=0}^{k+1} \At_{k+1-m} u_m = \sum_{m=0}^{k+1} \lambda_{k+1-m} u_m.
\end{equation}
For $m=k+1$, we have $\At_0 u_{k+1}=u_{k+1}$ on the left hand side and
$\lambda_0 u_{k+1} = u_{k+1}$ on the right hand side, so that the terms of order
$k+1$ compensate. Taking aside the terms of order $m=k$ leads to the equation:
\begin{equation}
\label{eq:ukeq}
(\A_1 +W - \lambda)u_k = \sum_{m=0}^{k-1} g_{k,m}, 
\qquad 
g_{k,m} = -\At_{k+1-m} u_m + \lambda_{k+1-m} u_m.
\end{equation} 
A sufficient condition for the solution $u_k$ to exist in $\Sw$ is that $g_{k,m}\in\Sw$ 
for $m\in\{0,\hdots,k-1\}$. For $m\in\{1,\hdots,k-1\}$, a sufficient condition for that is 
that $\At_{k+1-m} u_m$ has average~0 with respect to~$\nu_W$, which is 
clear from~\eqref{eq:invariance} and the fact that $u_m \in \Sw$. It therefore only remains 
to show that $g_{k,0}=-\At_{k+1}\hh + \lambda_{k+1}\hh$ belongs to $\Sw$. Two cases have to 
be distinguished here:
\begin{enumerate}[(a)]
\item if $k<p$, then $k+1\leq p$ and we can still use the invariance 
relation~\eqref{eq:invariance} applied to $\phi\equiv \hh$, along with the fact 
that $\lambda_{k+1}=a_{k+1}$:
\[
\intX g_{k,0}\, d\nuw = - \intX (\At_{k+1}\hh)\, d\nuw + \lambda_{k+1} \intX \hh\, d\nuw
= - a_{k+1}\intX \hh\, d\nuw + a_{k+1}\intX \hh\, d\nuw = 0.
\]

\item in the terminal case $k=p$, we cannot use~\eqref{eq:invariance} and $\lambda_{p+1}$
has a different expression (recall~\eqref{eq:lambdap1}). Let us compute this expression
explicitly. In view of~\eqref{eq:leading},
\[
\intX Wf\, d\nuw  =  \intX W f_0 \hw\, d\nu - \left( \intX f_0\, d\nuw \right) \left( \intX W \,d\nuw\right),
\]
and, given that $Wf_0\hw =\tilde{g} + \lambda\hw f_0 - \A_1^* (\hw f_0)$ 
and $\intX W\, d\nuw=\lambda$,
\[
\intX Wf\, d\nuw  =\intX \tilde{g}\, d \nu + \lambda \intX f_0 \hw\, d\nu 
- \intX \A_1^*(\hw f_0) d\nu
- \lambda \intX f_0 \hw\, d\nu.
\]
Since $\A_1\mathds{1}=0$, 
\[
\intX \A_1^*(\hw f_0) d\nu = \intX (\A_1\mathds{1}) \hw f_0\, d\nu = 0.
\]
Finally, using the expression of $\tilde{g}$ in~\eqref{eq:leading} and $\intX \hw\, d\nu=1$,
\[
\begin{aligned}
\intX Wf \,d\nuw & =\dps \intX \tilde{g}\, d\nu = - \intX (\At_{p+1})^*\hw \, d\nu + \intX \hw\, d\nu\, 
\frac{\dps \intX (\At_{p+1}\hh)\, d\nuw}
{\dps \intX \hh\, d\nuw}
\\ & \dps = - \intX \At_{p+1} \mathds{1} \,d\nuw + \frac{\dps \intX (\At_{p+1}\hh)\, d\nuw}
{\dps \intX \hh\, d\nuw}.
\end{aligned}
\]
From this calculation, we obtain, with~\eqref{eq:lambdap1},
\[
\lambda_{p+1}= \intX \At_{p+1}\mathds{1} d\nuw + \intX Wf\, d\nuw =
\frac{\dps \intX \At_{p+1}\hh \, d\nuw}
{\dps\intX \hh\, d\nuw},
\]
so that
\[
\intX g_{p,0}\, d\nuw = - \intX (\At_{p+1}\hh) d\nuw + \lambda_{p+1}\intX \hh\, d\nuw
= 0.
\]
\end{enumerate}
Therefore, for any $k\in\{0,\hdots,p\}$ and any $m\in\{0,\hdots,k-1\}$, it holds $g_{k,m}\in\Sw$. This allows to conclude that the equations~\eqref{eq:ukeq} are well-posed in $\Sw$ and~\eqref{eq:eigenQ} is satisfied.

We are now in position to conclude the proof. Inserting~\eqref{eq:eigenQ} in the stationarity equation~\eqref{eq:stationarity},
\[
\intX \Qdt^W\hwdt\, d\nuwdt= \intX\left( \e^{\Dt \lambdatdt}\hwdt + \Dt^{p+2}r_{W,\Dt}\right) d\nuwdt
= \e^{\Dt \lambdadt} \intX \hwdt\, d\nuwdt,
\]
so that
\[
\e^{\Dt \lambdadt}= \e^{\Dt \lambdatdt} 
+ \Dt^{p+2} \frac{\dps \intX r_{W,\Dt} \,d\nuwdt}{\dps \intX \hwdt\, d\nuwdt}.
\]
At this stage, it suffices to prove that the remainder term is uniformly of order
$\Dt^{p+2}$ for $\Dt$ sufficiently small. We note to this end that
 $\hwdt= \hh + \Dt\, u_1 + \hdots + \Dt^p u_p$, where the functions $u_1,\dots,u_p$ are regular and
$\hh >0$. Given that the state space $\X$ is compact, there exists $\varepsilon >0$ such that
$\hh \geq \varepsilon > 0$. This implies in particular that there exists $\Dt'>0$ 
such that, for any $0<\Dt \leq \Dt'$, it holds $\hwdt \geq \varepsilon/2>0$. We also
know that there exists $\Dt^*>0$ and $C>0$ such that, for any $0<\Dt \leq \Dt^*$, 
it holds $\| r_{W,\Dt} \|_{\Linfty} \leq C$. As a result, for $0<\Dt \leq \min(\Dt',\Dt^*)$,
\[
 \left| \frac{\dps \intX r_{W,\Dt}\, d\nuwdt}{\dps \intX \hwdt\, d\nuwdt}\right|
\leq \frac{\dps \intX |r_{W,\Dt}|\, d\nuwdt}{\dps \intX \hwdt\, d\nuwdt} \leq 
\frac{2C}{\varepsilon},
\]
which gives the claimed result. 

\subsubsection{Proof of Lemma~\ref{lem:approxoperator}}
\label{sec:approxoperator}

We follow the strategy outlined in~\cite{leimkuhler2016computation,lelievre2016partial}, which uses a truncated inverse 
series expansion. The first step is to use the expansion of the 
eigenvalue $\e^{\Dt \lambdatdt}$ as in the proof of Lemma~\ref{lem:step3}:
\[
\begin{aligned}
\e^{\Dt \lambdatdt} &= \intX \Qdt^W\mathds{1} (1 +\Dt^p f)\, d\nuw
\\ &= \dps \intX \left(1 + \Dt \At_1\mathds{1} + 
\Dt^2 \At_2\mathds{1} +
\hdots 
+ \Dt^{p+1} \At_{p+1}\mathds{1} + \Dt^{p+2}\Rm_{W,\Dt} \mathds{1}\right)
   (1+\Dt^p f)\,d\nuw 
\\ 
&= 1 + \Dt \lambda + \Dt^2 \lambda_2 + \hdots + \Dt^p \lambda_p
+ \Dt^{p+1} \lambda_{p+1} + \Dt^{p+2} r_{W,\Dt},
\end{aligned}
\]
where the coefficients $\lambda_m$ are defined in~\eqref{eq:lambdam1}-\eqref{eq:lambdap1}, and
there exists $C>0$ such that $|r_{W,\Dt}| \leq C$ for $0<\Dt\leq \Dt^*$. 
This expression, combined with the expansion~\eqref{eq:dlsemigroupgeneral} of $\Qdt^W$ leads to:
\[
\label{eq:operatorAB}
\Pi_W \left( \frac{\Qdt^W - \e^{\Dt \lambdatdt}}{\Dt} \right) \Pi_W = A + \Dt B_{\Dt} + \Dt^{p+1}
R_{W,\Dt}, 
\]
with
\[
\label{eq:defAB}
A= \Pi_W (\A_1 + W - \lambda) \Pi_W, \quad B_{\Dt} =\Pi_W(\At_2 - \lambda_2) \Pi_W + \hdots 
+ \Dt^{p-1}\Pi_W ( \At_{p+1} - \lambda_{p+1} ) \Pi_W.
\]
The operator~$A$ is invertible on $\Sw$ by Assumption~\ref{as:stability}.
Now we are back to the setting of~\cite{leimkuhler2016computation,lelievre2016partial} and 
it suffices to write the formal series expansion of the inverse of 
$A+\Dt B_{\Dt}=(\Id +\Dt B_{\Dt}A^{-1})A$ up to order $p$ by setting 
\[
\widetilde{S}_{\Dt}^W = A^{-1}\sum_{n=0}^p (-1)^n \left(B_{\Dt} A^{-1}\right)^n,
\] 
and then only retaining the terms of order at most $\Dt^{p+1}$ in this expression. More
precisely, denoting $C_k=\Pi_W (\At_k - \lambda_k) \Pi_W$, we find
\[
S_{\Dt}^W=A^{-1} - \Dt A^{-1}C_2A^{-1} + \Dt^2\left( A^{-1}C_2A^{-1}C_2A^{-1} - A^{-1}C_3A^{-1}
\right) +\Dt^3 \mathcal{C}_3 + \hdots + \Dt^p \mathcal{C}_p,
\]
where the operators $\mathcal{C}_k$ are defined using the operators $C_k$ and $A^{-1}$.
The operator $S_{\Dt}^W$ is well defined and leaves $\Sw$ invariant since each 
$\mathcal{C}_k$ consists in a finite number of applications of operators of the form $C_kA^{-1}$ and a final application of
$A^{-1}=\Pi_W (\A_1 + W - \lambda)^{-1} \Pi_W$. It is then easy to check that, by construction, 
the equality~\eqref{eq:approxoperator} is satisfied.

\subsection{Proof of Proposition~\ref{prop:directTU}}
\label{sec:proof_directTU}

We first show that, if $\Qdt^W$ satisfies Assumption~\ref{as:minorization} with a reference
probability measure~$\eta$, then $\Qtdt^W$ satisfies Assumption~\ref{as:minorization} with the 
same measure $\eta$. By
Assumption~\ref{as:minorization}, there exist $\varepsilon>0$ and a measure $\eta\in\PX$ such that, for any
bounded measurable nonnegative function~$\varphi$,
\begin{equation}
\label{eq:intineq}
\varepsilon \eta(\varphi) \leq \Qdt^W\varphi \leq \varepsilon^{-1} \eta(\varphi),
\end{equation}
so that, applying $\Udt^W$ on the right of~$\Qdt^W$ and $\Tdt^W$ on the left, 
\[
\varepsilon  \eta(\Udt^W \varphi) \Tdt^W\mathds{1} \leq \Tdt^W \Qdt^W\Udt^W\varphi  
\leq \varepsilon^{-1} \eta( \Udt^W \varphi) \Tdt^W\mathds{1}.
\]
Using~\eqref{eq:minoTUW} leads to
\[
\varepsilon\alpha^2  \eta( \varphi) \leq  \Qtdt^W\varphi
\leq \alpha^{-2} \varepsilon^{-1} \eta(\varphi),
\]
so that $\Qtdt^W$ satisfies Assumption~\ref{as:minorization}. In view of 
Theorem~\ref{theo:feynmankacdiscr}, the scheme $\Qtdt^W$ admits a unique invariant probability measure
$\nuwtdt$ and an eigenvalue $\lambdatdt$ defined by~\eqref{eq:lambdatdt} . Now, 
integrating~\eqref{eq:intineq} with respect to~$\nuwdt$ and using~\eqref{eq:stationarity} gives
\[
\varepsilon \eta(\varphi) \leq\ \e^{\Dt \lambdadt}\nuwdt(\varphi) \leq \varepsilon^{-1} \eta(\varphi).
\]
The same reasoning holds for $\nuwtdt$. There exists therefore $\varepsilon'>0$ for which the following inequalities hold in the sense of positive measures:
\begin{equation}
\label{eq:measurebounds}
\varepsilon' \eta \leq \nuwdt \leq \frac{1}{\varepsilon'} \eta, \qquad \varepsilon' \eta \leq \nuwtdt \leq \frac{1}{\varepsilon'} \eta.
\end{equation}

We are now in position to prove the equality of the eigenvalues $\lambdadt$ and 
$\lambdatdt$ defined respectively by~\eqref{eq:lambdadt} and~\eqref{eq:lambdatdt}.
From~\eqref{eq:stationarity}, it holds, for any $\varphi \in\S$, 
\[
\intX \left(\Qdt^W\right)^n\varphi \, d\nuwdt = \left( \intX \Qdt^W\mathds{1}\, d\nuwdt\right) \left( \intX \left(\Qdt^W\right)^{n-1}\varphi\, d\nuwdt\right) = \left( \intX \Qdt^W\mathds{1}\, d\nuwdt\right)^n \left( \intX \varphi\, d\nuwdt\right).
\]
Applying this last relation to $\Udt^W\varphi$ for $\varphi\in\S$ and using
the definition of $\lambdadt$,
\[
\intX \left(\Qdt^W\right)^n U_{\Dt}^W \varphi \, d\nuwdt = \left( \intX \Qdt^W\mathds{1}\, d\nuwdt\right)^n \left( \intX  U_{\Dt}^W\varphi\, d\nuwdt\right)
=\e^{n\Dt\lambdadt}  \intX  U_{\Dt}^W\varphi\, d\nuwdt.
\]
Similarly,
\[
\intX \left(\Qtdt^W\right)^n\varphi \, d\nuwtdt = \left( \intX \Qtdt^W\mathds{1}\, d\nuwtdt\right)^n \left( \intX \varphi\, d\nuwtdt\right)
=\e^{n\Dt\lambdatdt}  \intX  \varphi\, d\nuwtdt.
\]
It then follows that, for any positive $\varphi\in\S$,
\begin{equation}
\label{eq:exp_n_lambda}
\e^{n\Dt (\lambda_{\Dt}-\lambdatdt)} = \frac{\dps \intX \left(\Qdt^W\right)^n U_{\Dt}^W\varphi \, d\nuwdt}{\dps \intX \left(\Qtdt^W\right)^n\varphi \, d\nuwtdt} \times \frac{\dps \intX \varphi\, d\nuwtdt}{\dps \intX  U_{\Dt}^W\varphi\, d\nuwdt} = \frac{\dps \intX \left(\Qdt^W\right)^n U_{\Dt}^W\varphi \, d\nuwdt}{\dps \intX T_{\Dt}^W \left(\Qdt^W\right)^n U_{\Dt}^W \varphi \, d\nuwtdt} \times \frac{\dps \intX \varphi\, d\nuwtdt}{\dps \intX  U_{\Dt}^W\varphi\, d\nuwdt}.
\end{equation}
It remains to note that the right hand side of~\eqref{eq:exp_n_lambda} is uniformly 
bounded in~$n$. Indeed, denoting by $\psi_n=\left(\Qdt^W\right)^n U_{\Dt}^W\varphi$ 
for a positive $\varphi\in\S$, we obtain using~\eqref{eq:minoTUW} and~\eqref{eq:measurebounds}:
\[
0 \leq \frac{\dps \intX \psi_n \, d\nuwdt}{\dps \intX T_{\Dt}^W \psi_n \, d\nuwtdt} 
\leq  \frac{\dps \intX \psi_n \, d\nuwdt}{\dps \intX\alpha \psi_n \, d\nuwtdt} 
\leq \frac{\dps\intX \psi_n (\varepsilon')^{-1} d\eta}
{\dps \alpha \intX \psi_n \varepsilon' d\eta}
\leq  \frac{1}{\alpha (\varepsilon')^2},
\]
this bound being independant of $n$. Similarly, 
\[
0\leq \frac{\dps \intX \varphi \, d\nuwtdt}{\dps \intX U_{\Dt}^W \varphi \, d\nuwdt} 
\leq \frac{1}{\alpha (\varepsilon')^2}.
\]
Therefore, the right-hand side of~\eqref{eq:exp_n_lambda} is uniformly bounded for all 
$n \geq 0$, which proves that $\lambdadt \leq \lambdatdt$ by taking the limit
$n\to+\infty$. A similar reasoning leads to $\lambdatdt \leq \lambdadt$, hence
$\lambdadt = \lambdatdt$.

\subsection*{Acknowledgments}
The authors are grateful to Mathias Rousset for his help in understanding Feynman--Kac models.
We also thank Frédéric Cérou, Jonathan Mattingly, Julien Roussel and Hugo Touchette, Jonathan
Weare for fruitful discussions, and the anonymous referees for their useful comments. The
PhD fellowship of Gr\'egoire Ferr\'e is
partly funded by the Bézout Labex, funded by ANR, reference ANR-10-LABX-58. 
 The work of Gabriel Stoltz was funded in part by the Agence Nationale de la Recherche,
under grant ANR-14-CE23-0012 (COSMOS) and by the European Research Council under the European
Union's Seventh Framework Programme (FP/2007-2013)/ERC Grant Agreement number 614492. We also
benefited from the scientific environment of the Laboratoire International Associ\'e between the
Centre National de la Recherche Scientifique and the University of Illinois at Urbana-Champaign.

\appendix

\section{Markov contractions and Dobrushin coefficients}
\label{sec:dobrushin}

Denoting by $\MX$ is the set of measures over $\X$, we define $\Mz=\{\eta \in\MX \, |\, \eta(\X)=0\}$ the
set of (unsigned) measures with zero mass. The contraction norm of a Markov operator $Q:\PX\to\PX$ is
\[
\label{eq:contractrioncoeff}
\vertiii {Q}  := \underset{\eta \in \Mz }{\sup} \frac{ \| \eta Q \|\TV }
{\| \eta \|\TV } = \underset{\mu,\nu\in\PX}{\sup} \frac{ \| \mu Q -\nu Q \|\TV }
{\| \mu - \nu \|\TV },
\]
the second equality coming from the fact that all elements in $\Mz$ are proportional to the difference of two probability measures. In particular,
\[
\| \mu Q - \nu Q\|\TV \leq \vertiii {Q} \, \| \mu - \nu \|\TV.
\]
A fundamental tool~\cite{del2000branching,del2004feynman,del2001stability} 
for the study of Feynman--Kac type semigroups~\eqref{eq:generaldiscr} and introduced
by Dobrushin~\cite{dobrushin1956centralI,dobrushin1956centralII} is the so-called
 Dobrushin ergodic coefficient, which can be defined for a Markov operator $Q$ as:
\begin{equation}
\label{eq:dobrushin}
\alpha(Q) = \inf_{\substack{ q,q'\in\X\\ \{A_i\}_{1 \leq i \leq m}\subset\X }} \left\{ \sum_{i=1}^m \min \left(  Q(q,A_i), Q(q',A_i)\right) \right\},
\end{equation}
where the infimum in the last equality runs over points $q,q'\in\X$ and all 
partitions $(A_i)_{i=1}^m$ of $\X$. If we interpret $Q(q,A_i)$ as the probability of going 
from $q$ into the set $A_i$, we see that this coefficient provides information on the mixing 
properties of the operator $Q$. The link between this coefficient and the contraction properties 
of $Q$ is made precise by the following 
relationship~\cite{dobrushin1956centralI,dobrushin1956centralII}:
\begin{equation}
\label{eq:contractiondobrushin}
\vertiii{Q} = 1 - \alpha(Q).
\end{equation}
As a result, a minorization condition on $Q$ translates into a contraction of the operator through
its ergodic coefficient $\alpha(Q)$. Relation~\eqref{eq:contractiondobrushin} is essentially
obtained by a Hahn decomposition of measures of zero mass, as made precise 
in~\cite{dobrushin1956centralI,dobrushin1956centralII}. 

\bibliographystyle{abbrv}

\end{document}